\documentclass[11pt,reqno,a4paper]{amsart} 

\makeatletter
\def\@tocline#1#2#3#4#5#6#7{\relax
	\ifnum #1>\c@tocdepth 
	\else
	\par \addpenalty\@secpenalty\addvspace{#2}%
	\begingroup \hyphenpenalty\@M
	\@ifempty{#4}{%
		\@tempdima\csname r@tocindent\number#1\endcsname\relax
	}{%
		\@tempdima#4\relax
	}%
	\parindent\z@ \leftskip#3\relax \advance\leftskip\@tempdima\relax
	\rightskip\@pnumwidth plus4em \parfillskip-\@pnumwidth
	#5\leavevmode\hskip-\@tempdima
	\ifcase #1
	\or\or \hskip 1em \or \hskip 2em \else \hskip 3em \fi%
	#6\nobreak\relax
	\dotfill\hbox to\@pnumwidth{\@tocpagenum{#7}}\par
	\nobreak
	\endgroup
	\fi}
\makeatother

\usepackage[utf8]{inputenc}
\usepackage{amsmath}
\usepackage{amsthm}
\usepackage{amssymb}
\usepackage{euscript}
\usepackage{mathrsfs}
\usepackage{wasysym}
\usepackage[all]{xy}
\usepackage{graphicx}
\usepackage{color}
\usepackage{multirow}
\usepackage{pifont}
\usepackage[table]{xcolor}
\usepackage{tikz-cd}
\usepackage{enumitem}
\usepackage{xcolor}
\usepackage{verbatim}
\usepackage{tikz-cd}
\usepackage[hidelinks]{hyperref}

\usepackage[babel]{csquotes}        
\usepackage[maxnames=99,backend=bibtex,firstinits=true,url=false,doi=false,isbn=false]{biblatex}   
\AtBeginBibliography{\scriptsize}
\usepackage{chngcntr}
\counterwithout*{equation}{section}
\counterwithin{equation}{section}

\numberwithin{figure}{section}

\newtheorem*{mthm1}{Theorem 1}
\newtheorem*{mthm2}{Theorem 2}
\newtheorem{theorem}{Theorem}[section]
\newtheorem{corollary}[theorem]{Corollary}
\newtheorem{lemma}[theorem]{Lemma}
\newtheorem{proposition}[theorem]{Proposition}

\newtheorem{notation}[theorem]{Notation}
\newtheorem{definition}[theorem]{Definition}

\theoremstyle{definition}

\newtheorem{example}[theorem]{Example}

\newtheorem{remark}[theorem]{Remark}

 \newcommand{\N}{{\mathbb N}}
 \newcommand{\R}{{\mathbb R}}
\newcommand{\Q}{{\mathbb Q}}

 \newcommand{\D}{{\mathbb D}}
 \newcommand{\Proy}{{\mathbb P}}

\newcommand{\Ii}{\mc{I}}

\newcommand{\sign}{\operatorname{sign}}

\newcommand{\id}{\operatorname{id}}

\newcommand{\cinfty}{{\EuScript{C}^\infty}}

\newcommand{\mc}{\mathcal}

\newcommand{\mscr}{\mathscr}

\newlist{mydescription}{description}{1}
\setlist[mydescription]{font=\normalfont\normalcolor\textsc, before*={\normalfont}}

\begin{document}
	
	\title[]{Algebraic realization of stable Poincaré-Reeb graphs} 
	
\author{Enrico Savi}
\address{Univ Angers, CNRS, LAREMA, SFR MATHSTIC, F-49000 Angers, France}
\email{enrico.savi@univ-angers.fr}
\thanks{The author is supported by GNSAGA of INDAM, by the ANR NewMIRAGE (ANR-23-CE40-0002) and has been supported by ANR and IdEX of Université Côte d'Azur (IDEX UCAJedi ANR-15-IDEX-01) during the development of this work.}

\begin{abstract}
	We introduce the notion of domain of finite type $\mscr{D}\subset\R^n$ generalizing an earlier work of Bodin, Popescu-Pampu and Sorea. Then, we prove that every finite graph admitting a good orientation whose vertices have degree 1 or 3 can be realized as the Poincaré-Reeb graph of a stable (globally) algebraic domain of finite type $\mscr{D}\subset\R^n$, for every $n\geq 2$.  If in addition $n\geq 3$, we construct a class of graphs allowing vertices of degree $2$ also. Algebraic approximation techniques à la Nash-Tognoli and stable Morse functions are fundamental tools in our approach. In particular, the recent extensions over $\Q$ of such algebraic approximation techniques developed by Ghiloni and the author allow us to reduce the coefficients of the describing polynomials over $\Q$ and to extend our constructions over real closed fields.
\end{abstract}


\date{\today}

\maketitle


\tableofcontents


\section*{Introduction}

\subsubsection*{Reeb graphs of smooth functions on manifolds} The main topic discussed in this article is Reeb graphs, often referred to in the literature as Kronrod-Reeb graphs or Poincaré-Reeb graphs. Throughout this article, we will refer to these objects as Reeb graphs when considering graphs associated with smooth functions on a smooth manifold and as Poincaré-Reeb graphs when considering graphs associated with domains $\mscr{D}\subset\R^n$ of finite type, generalizing the case introduced and studied in \cite{BPS24}. In this way, we follow the main notation that appears in the literature for such objects and distinguish them when we relate these notions.

After their first appearance due to Reeb \cite{Ree46}, whereas Poincaré \cite{Poi10} introduced similar concepts already in 1904, Reeb graphs proved to be extremely important invariants in the study of Morse functions and, subsequently, for their concrete applications in mathematical modeling, see \cite{BGSF08}. We introduce its first definition in a more general context. Let $X$ be a topological space and $f:X\to \R$ a continuous function. We define the \emph{Reeb space} $\widetilde{X}_ {f}$ of $f$ as the quotient topological space of $X$ with respect to the equivalence relation $\sim$ on $X$ defined as: $x\sim y$ if $f(x)=f(y)=z$ and $x$ and $y$ belong to the same connected component of $f^{-1}(z)$. Therefore, the topological space $\widetilde{X}_f$ encodes topological information about the levels of $f:X\to\R$. Usually, we assume that $X$ and $f$ have an additional structure to ensure that the topological space $\widetilde{X}_f$ is tractable. For the purpose of this paper, we are mostly interested in the case where $X$ is a smooth manifold and $f:X\to \R$ is a Morse function. There is an extensive literature on Reeb spaces of Morse functions; in this article, we will focus on some aspects of the theory. It is important to note that several extensions of this concept have recently been made in different directions, by Gelbukh, who, in a series of papers, extended various results to Morse-Bott functions \cite{Gel24} and conducted a more in-depth study of the topology of Reeb spaces \cite{Gel18}, and by Saeki \cite{Sae22,Sae24}, who extended the study of Reeb spaces to the case of smooth functions on compact smooth manifolds having a finite number of critical values. There are also extensions of these concepts to singular spaces such as in semi-algebraic geometry, or more generally o-minimal geometry, over arbitrary real closed fields, see \cite{PS07,BCP22}. In all these contexts, what allows us to  understand well Reeb spaces is that they admit a graph structure. Therefore, in all the above situations, we refer to the \emph{Reeb graph} $\mathcal{R}(f)$ of a function $f:X\to \R$ as the Reeb space $\widetilde{X}_f$ endowed with its graph structure. We specify that in all the previous contexts, graph actually means \emph{multigraph}. We will maintain this notation throughout the paper.

Then, a natural question arises: which classes of graphs can be obtained as Reeb graphs of a smooth function $f:M\to \R$, where $M$ is a compact smooth manifold? A first complete answer is given by Sharko \cite{Sha06} in the case of Morse functions on surfaces and then generalized by Michalak \cite{Mic18} in all dimensions. In both cases, a graph $\Gamma:=(V,E)$ is realizable as a Reeb graph of a smooth function on a compact smooth manifold $M$ if and only if $\Gamma$ admits a good orientation, see Definition \ref{def:orientation} below and \cite[Thm.4.3]{Mic18}. Observe that in the previous results the manifold $M$ is not fixed in advance: when fixing $M$ the problem of classifying Morse functions on $M$ becomes harder, a very non-trivial description appeared already in the case of the 2-dimensional sphere \cite{Nic08}. 

One can then require further regularity of $M$ and of the Morse function $f:M\to \R$, for example that $M$ be a nonsingular real algebraic set and that $f:M\to \R$ be a regular map, i.e., a global quotient of polynomials with non-vanishing denominator. As a consequence of the techniques introduced in \cite{GSa,Sav1}, a general answer follows in the case of graphs of degree $1$ and $3$. In fact, let $\Gamma:=(V,E)$ be a graph with a given good orientation. By \cite[Thm.4.3]{Mic18}, there exists a Morse function $f:M\to \R$ that realizes $\Gamma$ as its Reeb graph $\mathcal{R}(f)$, and we can further assume that $f$ has distinct critical values. Then, applying \cite[Thm.3.9]{GSa}, there exists a $\Q$-nonsingular $\Q$-algebraic set $M'$, a $\cinfty$ diffeomorphism $\phi: M\to M'$, and a $\Q$-regular function $g:M'\to\R$ such that $g\circ\phi$ is arbitrarily close to $f$ in $\cinfty(M)$. Here and throughout the paper we consider the space of smooth maps $\cinfty(M,\R^m)$ endowed with the weak smooth topology, for more details see \cite[\S 1\;\&\;2]{Hir94}. Therefore, since Morse functions with distinct critical values are stable in the sense of \cite{Mat69}, we directly obtain that $\mathcal{R}(g)$ also realizes $\Gamma$. Moreover, if $M\subset\R^n$ is a smooth hypersurface, then $M'$ as above can be obtained as a $\Q$-nonsingular $\Q$-algebraic hypersurface of $\R^n$ by a simplified argument of \cite[Thm.3.9]{GSa}, for details look at the proof of Theorem \ref{thm:approx} below. For detailed definitions of real $\Q$-algebraic geometry introduced in \cite{FG,GSa}, in the special case of hypersurfaces, see Section \ref{sub:1.4}. These last remarks will be crucial for our results.

\subsubsection*{Poincaré-Reeb graphs of domains of finite type and their realization}

In this article, we are primarily interested in a slightly different situation than in the previous context. In 2024, Bodin, Popescu-Pampu, and Sorea \cite{BPS24} originally introduced the notions of domain of finite type $\mscr{D}\subset\R^2$, Poincaré-Reeb graph $\mathcal{R}(\mscr{D})$ of $\mscr{D}$, and studied the topological properties of $\mscr{D}$ encoded by its Poincaré-Reeb graph. Their main result \cite[Thm.3.5]{BPS24} consists in producing an algebraic realization of graphs with vertices of degree $1$ or $3$. More precisely, the result states that every graph of degree $1$ and $3$ can be realized as the Poincaré-Reeb graph $\mathcal{R}(\mscr{D})$ of a domain of finite type $\mscr{D}\subset\R^2$ whose boundary $\partial \mscr{D}$ consists of a (finite) union of  nonsingular connected components of an algebraic curve in $\R^2$. This latter paper is the main motivation for the results presented in this article and a source of inspiration for generalizing the constructions in different directions.

In this article, we introduce the notion of a \emph{domain of finite type} $\mscr{D}\subset\R^n$ as a natural generalization of the original definition in \cite{BPS24}. More precisely, $\mscr{D}\subset\R^n$ is an embedded manifold with boundary such that the restriction on $\partial\mscr{D}$ of the projection $\pi:\R^n\to \R$ on the first coordinate is a Morse function with a finite number of critical values and $\pi|_{\mscr{D}}$ is proper. Then, in Section \ref{sec:1}, we define the notions of \emph{Poincaré-Reeb space} $\widetilde{\mscr{D}}$ and \emph{Poincaré-Reeb graph} $\mathcal{R}(\mscr{D})$ of $\mscr{D}$ which correspond, in the above notation, to the Reeb space and the Reeb graph of the restriction to $\mscr{D}$ of the projection onto the first coordinate, respectively, and we study the topological information encoded by these invariants. Some of these properties are deduced after relating the Poincaré-Reeb graph of a compact domain of finite type $\mscr{D}\subset\R^n$ to the Reeb graph of a smooth function on a compact smooth hypersurface $M\subset\R^{n+1}$ that we call the \emph{double} of $\mscr{D}$. Then, we show that the Poincaré-Reeb graph $\mathcal{R}(\mscr{D})$ of a compact domain of finite type $\mscr{D}\subset\R^n$ do admit a \emph{good orientation} in the sense of \cite{Mic18}, for every $n\geq 2$, see Lemma \ref{cor:orientation}. This highlights an inaccuracy in \cite{BPS24}: the main result proven applies to graphs that admit a good orientation, not to every graph $\Gamma$ having vertices of degree $1$ and $3$ as stated. We outline that the proof of \cite[Thm.3.5]{BPS24} implicitly makes use the notion of good orientation, see \cite[Fig.14]{BPS24}. See Figure \ref{fig:ex1} for an example of a graph with only vertices of degree $3$ that cannot be realized as the Poincaré-Reeb graph of any domain of finite type, as it does not have any vertex of degree $1$. Next, we define the notion of a \emph{stable} domain of finite type by requiring that the function $\pi|_{\partial{\mscr{D}}}$ be a stable Morse function in the sense of \cite{Mat69}, i.e., a Morse function with distinct critical values. For further information on stable mappings, see \cite{GG73}. Finally, we extend the notion of algebraic domain of finite type in a similar way to \cite{BPS24}, but then we will mainly focus on a stronger notion that we call \emph {globally algebraic domain of finite type} $\mscr{D}\subset\R^n$, i.e., a domain of finite type $\mscr{D}\subset\R^n$ whose boundary $\partial\mscr{D}$ is a nonsingular algebraic hypersurface of $\R^n$.

Section \ref{sec:2} is devoted to our approximation results. In particular, since we apply the approximation techniques developed in \cite{GSa,Sav1}, we will obtain equations over $\Q$. More precisely, our main approximation result is as follows:

\begin{mthm1}\label{mthm1}
	Every stable compact domain of finite type $\mscr{D}\subset\R^n$ is vertically equivalent to a globally algebraic stable domain of finite type $\mscr{D}'\subset\R^n$. In addition, we may assume that $\partial\mscr{D}'\subset\R^n$ is a projectively $\Q$-closed $\Q$-nonsingular $\Q$-algebraic hypersurface arbitrarily $\cinfty$ close to $\partial\mscr{D}$, that is:
	\begin{enumerate}[label=\emph{(\roman*)}, ref=(\roman*)]
		\item there exists $q\in\Q[x_1,\dots,x_n]$ such that $\mscr{D}'=\{x\in\R^n\,|\,q(x)\geq 0\}$, $\partial\mscr{D}'=\mathcal{Z}_{\R^n}(q)$ and $\nabla q(x)\neq \overline{0}$ for every $x\in\partial\mscr{D}'$.
		\item $\partial\mscr{D}'\subset\R^n$ coincides with its projective closure in $\mathbb{P}^n(\R)$.
	\end{enumerate}
\end{mthm1}

Therefore, Theorem \ref{mthm1} guarantees that, after constructing stable compact domains of finite type $\mscr{D}\subset\R^n$, it is always possible to obtain a globally algebraic stable domain of finite type $\mscr{D}'\subset\R^n$ that is vertically equivalent to $\mscr{D}$, so in particular $\mathcal{R}(\mscr{D})\cong\mathcal{R}(\mscr{D}’)$. Therefore, the problem of algebraically realizing stable domains of finite type with a given Poincaré-Reeb graph reduces to the construction of stable domains of finite type $\mscr{D}\subset\R^n$ with a fixed Poincaré-Reeb graph. Thus, combining the results in \cite{Mic18} and our doubling technique in Proposition \ref{prop:hom-equiv} with the application of Theorem 1 above, we obtain the following result:

\begin{mthm2}\label{thm:2}
	Let $\Gamma:=(V,E)$ be a finite graph (without unbounded edges) whose vertices have degree $1$ or $3$ with $|\Gamma|\subset\R^n$ such that $\pi|_{|\Gamma|}:|\Gamma|\to \R$ is a good orientation with $\pi(v)\neq\pi(w)$ for every $v,w\in V$ with $v \neq w$. Then, there exists a globally algebraic domain $\mscr{D}\subset\R^n$ whose Poincaré-Reeb graph $\mathcal{R}(\mscr{D})$ is isomorphic to $\Gamma$. In addition:
	\begin{enumerate}[label=\emph{(\roman*)}, ref=(\roman*)]
		\item $\partial\mscr{D}$ is a projectively $\Q$-closed $\Q$-nonsingular $\Q$-algebraic hypersurface of $\R^n$.
		\item $\mscr{D}$ is homotopically equivalent to the Poincaré-Reeb space $\widetilde{\mscr{D}}$.
	\end{enumerate}
\end{mthm2}

We emphasize that Theorem \ref{thm:2} consists of a generalization of \cite[Thm.3.6]{BPS24} even in the case $n=2$. In fact, we ensure that $\partial\mscr{D}\subset\R^2$ is a nonsingular algebraic curve which also coincides with its projective closure in $\mathbb{P}^2(\R)$, and not just a union of nonsingular connected components of an algebraic curve. Note in particular that the domain $\mscr{D}\subset\R^n$ has a particularly simple structure since it coincides with the locus of positivity of a nonsingular polynomial $q\in\Q[x]$. We then conclude the article by extending the result to allow some vertices of index $2$; see Theorem \ref{thm:Poincaré-Reeb2} below for a detailed statement. Finally, since we get polynomials with rational coefficients, we discuss in Remark \ref{rem:final} how to extend our results to any real closed field.

\vspace{1em}

\section{Poincaré-Reeb graphs of domains of finite type in $\R^n$}\label{sec:1}

\vspace{0.5em}

\subsection{Domains of finite type in $\R^n$ and their Poincaré-Reeb graphs}\label{sub:1.1} Let us introduce the definition of a domain of finite type of $\R^n$, for $n\geq2$.

\begin{definition}\label{def:dom}
	Let $\mscr{D}\subset\R^n$ be a closed $\cinfty$ submanifold of $\R^n$ of dimension $n$ with boundary $\partial\mscr{D}$ and denote by $\pi:\R^n\to\R$ the projection onto the first coordinate. We say that $\mscr{D}$ is a \emph{domain of finite type} if:
	\begin{enumerate}[label=\emph{(\roman*)}, ref=(\roman*)]
		\item\label{en:dom1} $\pi|_{\partial\mscr{D}}$ is a Morse function with finitely many critical values.
		\item\label{en:dom2} $\pi|_{\mscr{D}}$ is a proper map.
	\end{enumerate}
\end{definition}

We remark that for the case $n=2$ the original notion of a domain of finite type $\mscr{D}\subset\R^2$ in \cite[Definition 2.7]{BPS24} is set in the $\EuScript{C}^0$ category whereas our Definition \ref{def:dom} requires the boundary $\partial\mscr{D}$ to be a $\cinfty$ submanifold and $\pi|_{\mscr{D}}$ to be a Morse function. This is not a substantial difference when comparing the two definitions for $n=2$. Indeed, whenever we apply approximation techniques higher regularity (at least $\EuScript{C}^2$) is necessary to preserve the shape of critical points of $\pi|_{\partial\mscr{D}}$ even in \cite{BPS24}.

\begin{remark}\label{rem:connected}
	Let $\mscr{D}\subset\R^n$ be a domain of finite type, then:
	\begin{enumerate}[label=(\roman*), ref=(\roman*)]
		\item $\pi|_{\partial\mscr{D}}$ has a finite number of critical points.
		\item $\partial\mscr{D}\subset\R^n$ is a $\cinfty$ hypersurface having a finite number of connected components.
	\end{enumerate}
	Observe that by Morse's Lemma \cite[\S 6.1]{Hir94}, the set of critical points of $\pi|_{\partial\mscr{D}}$ is discrete. Since $\pi$ satisfies Definition \ref{def:dom}\ref{en:dom1}\&\ref{en:dom2} and $\partial{\mscr{D}}$ is closed in $\mscr{D}$, then $\pi|_{\partial\mscr{D}}$ is proper as well. This proves that the set of critical points of $\pi|_{\partial\mscr{D}}$ is a closed and discrete subset of a compact subset of $\mscr{D}$, hence it is finite.  Let $b_1,\dots,b_\ell\in\R$ be the critical values of $\pi|_{\partial\mscr{D}}$. Let $b_0:=-\infty$ and $b_{\ell+1}:=+\infty$. Since $\pi|_{\partial\mscr{D}}$ is proper and Morse, we know that $\pi|_{\partial\mscr{D}}^{-1}(]b_i,b_{i+1}[)\cong \pi|_{\partial\mscr{D}}^{-1}(a)\times ]b_i,b_{i+1}[$, for $a\in]b_i,b_{i+1}[$, has a finite number of connected components for every $i\in\{0,\dots,\ell\}$ . Since the number of critical points of $\pi|_{\partial\mscr{D}}$ is finite, we deduce that $\partial\mscr{D}\subset\R^n$ has a finite number of connected components.
\end{remark}

Let us introduce the notion of Poincaré-Reeb graph for a domain $\mscr{D}\subset\R^n$, for every $n\geq 2$.

\begin{definition}\label{def:P-R-space}
	Let $\mscr{D}\subset\R^n$ be a domain of finite type. We define the equivalence relation $\sim$ on $\mscr{D}$ as follows: for every $x,y\in\mscr{D}$ we say that $x\sim y$ if $x$ and $y$ lie in the same connected component of $\pi^{-1}(\pi(x))\cap\mscr{D}$. Denote by $\widetilde{\mscr{D}}:=\mscr{D}/\sim$ the quotient topological space of $\mscr{D}$ under the equivalence relation $\sim$ and by $\rho:\mscr{D}\to\widetilde{\mscr{D}}$ the quotient map. We refer to $\widetilde{\mscr{D}}$ as the Poincaré-Reeb space of $\mscr{D}$.
\end{definition}

Observe that the Poincaré-Reeb space $\widetilde{\mscr{D}}$ of $\mscr{D}\subset\R^n$ has a graph structure $\Gamma:=(V,E)$. Indeed, the Thom conjecture on the triangulation of smooth mappings proved by Shiota \cite{Shi00} ensures that the projection $\pi|_\mscr{D}$ is triangulable. Then, \cite{HS13} implies that the Stein factorization of $\pi|_{\mscr{D}}$, that is, the unique continuous function $\widetilde{\pi}:\widetilde{\mscr{D}}\to \R$ such that $\pi|_{\mscr{D}}=\widetilde{\pi}\circ\rho$, is triangulable as well. This directly implies that $\widetilde{\mscr{D}}$ is triangulable of dimension $1$, that is, it admits a graph structure. More concretely, by following the proof proposed by Sharko \cite[Lem.2.1]{Sha06},  $\widetilde{\mscr{D}}$ admits a graph structure $\Gamma:=(V,E)$ defined in the following way:

\begin{enumerate}[label=(\roman*), ref=(\roman*)]
	\item Define the set of vertices $V$ as the set of critical points of $\pi|_{\partial\mscr{D}}$. Define a partial order on $V$ as $v_1\prec v_2$ if and only if $\pi(v_1)<\pi(v_2)$.
	\item Define the set of edges $E$ as follows:
	\begin{enumerate}
		\item We add an edge $e$ between two vertices $v_1,v_2\in V$, with $v_1\prec v_2$, corresponding to each connected component of $\pi|_{\mscr{D}}^{-1}(]\pi(v_1),\pi(v_2)[)$ whose closure intersects both the connected component of $\pi^{-1}|_{\mscr{D}}(\pi(v_1))$ and $\pi^{-1}|_{\mscr{D}}(\pi(v_2))$ containing $v_1$ and $v_2$, respectively. 
		\item We add an unbounded edge $e$ with vertex $v$ for every unbounded connected component of $\mscr{D}$ of the form $\pi|_{\mscr{D}}^{-1}(]-\infty,\pi(v)[)$ or $\pi|_{\mscr{D}}^{-1}(]\pi(v),+\infty[)$ containing no critical points of $\pi|_{\partial\mscr{D}}$ and whose closure intersect the connected component of $\pi|_{\mscr{D}}^{-1}(v)$ containing $v$. 
		\item We add an unbounded edge $e$ for every unbounded connected component of $\mscr{D}$ containing no critical points of $\pi|_{\partial\mscr{D}}$. 
	\end{enumerate}
\end{enumerate}


The previous construction motivates the next definition.

\begin{definition}\label{def:PR-graph}
	Let $\mscr{D}\subset\R^n$ be a domain of finite type. The Poincaré-Reeb space $\widetilde{\mscr{D}}$ equipped with the graph structure described above is denoted by $\mathcal{R}(\mscr{D})$ and it is called the \emph{Poincaré-Reeb graph} of $\mscr{D}$.
\end{definition}

We extend some notions and results originally proven in \cite{BPS24} for domains of finite type $\mscr{D}\subset\R^2$. Our results clarify in the general case which topological information is encoded by the Poincaré-Reeb graph $\mathcal{R}(\mscr{D})$ of a domain of finite type $\mscr{D}\subset\R^n$. Sometimes the results will be a bit weaker since the case of domains of finite type $\mscr{D}\subset \R^2$ is very special: for every $z\in\R$ each connected component of $\pi|_{\mscr{D}}^{-1}(z)$ is an interval. By contrast, when $n>2$ we have infinitely many types of $\cinfty$ manifolds with boundary which may appear as connected components of $\pi|_{\mscr{D}}^{-1}(z)$ for a regular value $z\in\R$ and, of course, many different singular preimages whenever $z\in\R$ is a critical value.

\begin{remark}\label{rem:crit}
		Let $M\subset\R^n$ be a $\cinfty$ hypersurface such that $\pi|_M$ is a Morse function. Let $h\in\cinfty(\R^n)$ such that $\mathcal{Z}_{\R^n}(h)=M$ and $\nabla h(x)\neq \overline{0}$ for every $x\in M$. Then, $x\in M$ is a critical point of $\pi|_M$ if and only if $\frac{\partial h}{\partial x_i}(x)=0$ for every $i=2,\dots,n$.
\end{remark}

First of all let us describe the local behavior of the projection $\pi|_{\mscr{D}}$ at a critical point $p$ of $\pi|_{\partial\mscr{D}}$ and the local structure of the critical level $\pi^{-1}(\pi(p))\cap\mscr{D}$ at $\pi(p)$.

\begin{lemma}\label{lem:Morse}
	Let $\mscr{D}\subset\R^n$ be a domain of finite type and let $p\in\partial\mscr{D}$ be a critical point of $\pi|_{\mscr{D}}$. Then, there are neighborhoods $V$ of $p$ in $\R^n$, $U$ of the origin $\overline{0}$ in $\R^n$ and a diffeomorphism $\Phi:U\to V$ such that $\Phi(\overline{0})=p$ and
	\[
	\pi\circ\Phi(y_1,\dots,y_n)=\pi(p)+y_1+\sum_{i=2}^{n} \epsilon_i y_i^2,
	\]
	where $\epsilon_i\in\{\pm 1\}$ are uniquely determined by the index of $p$ as a critical point of $\pi|_{\partial\mscr{D}}$. 
	
	In particular, if the index of $p$ is $k\in\{0,\dots,n-1\}$, then the germ of $\pi^{-1}(\pi(p))\cap \mscr{D}$ at $p$ is homeomorphic to the cone $\textnormal{Cone}(\mathbb{D}^{n-k-1}\times\mathbb{S}^{k-1})$ of $\mathbb{D}^{n-k-1}\times\mathbb{S}^{k-1}$ at the origin.
\end{lemma}

\begin{proof}
	Let $h\in\cinfty(\R^n)$ such that $\mathcal{Z}_{\R^n}(h)=\partial\mscr{D}$, $\nabla h(x)\neq \overline{0}$ for every $x\in\partial\mscr{D}$ and $\mscr{D}:=\{x\in\R^n\,|\, h(x)\geq 0\}$. By Morse's lemma \cite[\S 6.1]{Hir94} there is a sufficiently small open ball $B^n(p,\delta)\subset \R^n$ centered at $p$ of radius $\delta$, an open neighborhood $U'$ of $\overline{0}':=(0,\dots,0)$ in $\R^{n-1}$ and a diffeomorphism $\psi:U\to \partial\mscr{D}\cap B^n(p,\delta)$ such that $\psi(\overline{0}')=p$ and
	\[
	(\pi|_{\partial\mscr{D}}\circ \psi)(y_2,\dots,y_n)=\pi(p)+\sum_{i=2}^n \epsilon_i y_i^2, 
	\]
	where $\epsilon_i\in\{\pm 1\}$ are uniquely determined. Denote by $y':=(y_2,\dots,y_n)$ the above coordinates of $\R^{n-1}$. Consider the normal vector field $\nu\in\cinfty(\partial\mscr{D},\R^n)$ defined as $\nu(x):=\frac{\nabla h(x)}{||\nabla h(x)||}$. Then, the $\cinfty$ map $\Psi:]-1,1[\times U\to \R^n$ defined as
	\[
	\Psi(y_1,y'):=\psi(y')+y_1\nu(\psi(y'))
	\] 
	is a diffeomorphism when restricted to its image $V:=\Psi(]-1,1[\times U')$. Observe that, by construction, $\Psi(\overline{0})=\psi(\overline{0}')=p$, $\Psi(\{0\}\times U')=\psi(U')=B^n(p,\delta)\cap \partial\mscr{D}$ and $\Psi([0,1[\times U')=\mscr{D}\cap V$, as $\mscr{D}=\{x\in\R^n\,|\,h(x)\geq 0\}$.
	
	Let us study the local structure of $\pi\circ\Psi\in\cinfty(]-1,1[\times U')$ at $\overline{0}$, which is a critical point for $\pi\circ\Psi$. As the projection $\pi:\R^n\to \R$ is linear and $\Psi|_{\{0\}\times U'}(0,y')=\psi(y')$ for every $y'\in U'$, we have:
	\[
		(\pi\circ\Psi)(y_1,y')=\pi(\psi(y'))+y_1\pi(\nu(\psi(y')))=\pi(p)+\sum_{i=2}^n \epsilon_i y_i^2 + y_1\omega(y'),
	\]
	where $\omega:=\pi\circ\nu\circ\psi\in\cinfty(U')$. Moreover, as $\psi(\overline{0}')=p$ and $p$ is a critical point for $\pi|_{\partial\mscr{D}}$, then:
	\[
	\omega(\overline{0}')=\pi(\nu(p))={\frac{\partial h}{\partial x_1}}(p)\Big/{||\nabla h}(p)||=\pm1.
	\]
	Indeed, by Remark \ref{rem:crit}, we have that $\frac{\partial h}{\partial x_i}(p)= 0$ for every $i\in\{2,\dots,n\}$, but $\nabla h(p)\neq \overline{0}$, thus $\frac{\partial h}{\partial x_1}(p)\neq 0$ and $||\nabla h(p)||=\big|\frac{\partial h}{\partial x_1}(p)\big|$. Thus, up to shrinking the neighborhood $U'$ of $\overline{0}'$ in $\R^{n-1}$ if necessary, we may suppose that $\omega(y')\neq 0$ for every $y'\in U'$. Then, define the smooth map $\phi\in\cinfty(]-1,1[\times U',\R^n)$ as $\phi(y_1,y'):=(y_1\omega(y'),y')$, which is a diffeomorphism when restricted to its image as $\omega(y')\neq 0$ for every $y'\in U'$. Then, after defining $U:=\phi(]-1,1[\times U')$ and $\Phi:=\Psi\circ\phi^{-1}$, we get the desired local structure of $\pi\circ\Phi$ as follows:
	\[
	\pi\circ\Phi(y_1,\dots,y_n)=\pi(p)+y_1+\sum_{i=2}^{n} \epsilon_i y_i^2.
	\]
	
	Let us describe the local structure of $\pi^{-1}(\pi(p))$ at $p$. As $V=\Phi(U)$ is an open neighborhood of $p$ in $\R^n$, then $\pi^{-1}(\pi(p))\cap\mscr{D}\cap V$ is homeomorphic via $\Phi$ to the set $\{(y_1,y')\in U\,|\,y_1+\sum_{i=2}^{n} \epsilon_i y_i^2=0, y_1\geq 0\}=\{(-\sum_{i=2}^{n} \epsilon_i y_i^2,y')\in U\,|\,\sum_{i=2}^{n} \epsilon_i y_i^2\leq 0\}$, which is the intersection of $U$ with the graph of the function $-\sum_{i=2}^{n} \epsilon_i y_i^2$ restricted to $S:=\{y'\in\R^{n-1}\,|\,\sum_{i=2}^{n} \epsilon_i y_i^2\leq0\}$. Therefore, there are $\delta'>0$ and a neighborhood $V'$ of $p$ in $V\subset\R^n$ such that $\pi^{-1}(\pi(p))\cap\mscr{D}\cap V'$ is homeomorphic to $S\cap\mathbb{D}^{n-1}(\overline{0}',\delta')$. Observe that
	\[
	S\cap\mathbb{D}^{n-1}(\overline{0}',\delta')=\bigg\{y'\in\mathbb{D}^{n-1}(\overline{0}',\delta')\,|\,\sum_{i=2}^{n} \epsilon_i y_i^2\leq0\bigg\}
	\]
	is homeomorphic to the cone $\textnormal{Cone}_p(S\cap\mathbb{S}^{n-2}(\overline{0}',\delta'))$ of $(S\cap\mathbb{S}^{n-2}(\overline{0}',\delta')$ at $p$, as $\sum_{i=2}^{n} \epsilon_i (t y_i)^2=t^2(\sum_{i=2}^{n} \epsilon_i y_i^2)\leq 0$ if and only if $\sum_{i=2}^{n} \epsilon_i y_i^2\leq 0$ for $y'\in\mathbb{S}^{n-2}$ and for every $t\in\R$. Let us show that $S\cap\mathbb{S}^{n-2}(0,\delta')$ is homeomorphic to $\mathbb{S}^{k-1}\times\mathbb{D}^{n-k-1}$, where $k$ denotes the index of $p$ as a critical point of the Morse function $\pi|_{\partial\mscr{D}}$. Up to reordering the $y_i$'s, we may suppose that $\sum_{i=2}^{n} \epsilon_i y_i^2=-\sum_{i=2}^{k+1} y_i^2+\sum_{i=s+2}^{n} y_i^2$. Then, by identifying $\R^{n-1}$ with $\R^{k}\times\R^{n-k-1}$ and denoting by $u:=(y_2,\dots,y_{k+1})$ and by  $v:=(y_{k+2},\dots,y_{n-1})$, we have $\sum_{i=2}^{n} \epsilon_i y_i^2:=-||u||^2+||v||^2$. As a consequence,
	\begin{align*}
	S\cap\mathbb{S}^{n-2}&=\Big\{(u,v)\in\R^{k}\times\R^{n-k-1}\,|\,-||u||^2+||v||^2\leq 0, ||u||^2+||v||^2=1\Big\}\\
											&=\Big\{(u,v)\in\R^{k}\times\R^{n-k-1}\,|\,||u||^2\geq \frac{1}{2}, ||u||^2+||v||^2=1\Big\}
	\end{align*}
	and therefore the continuous map $\varphi:S\cap\mathbb{S}^{n-2}\to \mathbb{R}^{k}\times\mathbb{R}^{n-k-1}$ defined as $\varphi(u,v):=\big(\frac{u}{||u||},\frac{v}{||u||}\big)$ is well defined. Observe that $w:=\frac{u}{||u||}\in\mathbb{S}^{k-1}$ and $z:=\frac{v}{||u||}\in \D^{n-k-1}$ as $||u||^2=\frac{1}{1+||z||^2}\geq \frac{1}{2}$ if and only if $||z||\leq 1$, that is,  $\varphi:S\cap\mathbb{S}^{n-2}\to \mathbb{S}^{k-1}\times\mathbb{D}^{n-k-1}$. Moreover, its inverse $\varphi^{-1}: \mathbb{S}^{k-1}\times\mathbb{D}^{n-k-1}\to S\cap\mathbb{S}^{n-2}$ is defined by $\varphi^{-1}(w,z)=\Big(\frac{w}{\sqrt{1+||z||^2}},\frac{z}{\sqrt{1+||z||^2}}\Big)$, which is continuous as well; this shows that $\varphi:S\cap\mathbb{S}^{n-2}\to\mathbb{S}^{k-1}\times\mathbb{D}^{n-k-1}$ is a homeomorphism. Therefore, the germ of $\pi^{-1}(\pi(p))\cap\mscr{D}$ at $p$ is homeomorphic to the cone $\textnormal{Cone}(\mathbb{S}^{k-1}\times\mathbb{D}^{n-k-1})$ of $\mathbb{S}^{k-1}\times\mathbb{D}^{n-k-1}$ at $\overline{0}'\in\R^{n-1}=\R^{k}\times\R^{n-k-1}$, as desired.
\end{proof}

Let us describe the first topological information encoded by the Poincaré-Reeb space $\widetilde{\mscr{D}}$ of a domain of finite type $\mscr{D}\subset\R^n$.

\begin{proposition}\label{prop:hom-equiv}
	Let $\mscr{D}\subset\R^n$ be a domain of finite type. If for every $y\in\R$ the $k$-th fundamental group of each connected component of $\pi|_{\mscr{D}}^{-1}(y)$ is trivial for every $1\leq k\leq n-1$, then $\mscr{D}$ is homotopically equivalent to $\widetilde{\mscr{D}}$. 
	
\end{proposition}

\begin{proof}
	Recall that $\mscr{D}$ has a finite number of connected components $\mscr{D}_1,\dots,\mscr{D}_k$, for some $k\geq 1$, by Remark \ref{rem:connected}.  Let us prove that $\rho(\mscr{D}_1),\dots,\rho(\mscr{D}_k)$ are the connected components of $\widetilde{\mscr{D}}$. Since $\rho$ is surjective, since each $\mscr{D}_i$ is both open and closed in $\mscr{D}$ and since the topology of $\widetilde{\mscr{D}}$ is defined by the images under $\rho$ of the $\rho$-saturated open subsets of $\mscr{D}$, we only have to prove that each $\mscr{D}_i$ is $\rho$-saturated, that is, $\rho^{-1}(\rho(\mscr{D}_i))=\mscr{D}_i$. Fix $i\in\{1,\dots,k\}$. Let $x\in\mscr{D}_i$ and  let $C$ be the connected component of $\pi^{-1}(\pi(x))$ containing $x$. Then $C=(C\cap\mscr{D}_i)\cup(C\cap\bigcup_{j\neq i}\mscr{D}_j)$, where $C\cap\mscr{D}_i\neq\varnothing$ implies $C\cap\bigcup_{j\neq i}\mscr{D}_j=\varnothing$ as $C$ is connected and both $\mscr{D}_i$ and $\bigcup_{j\neq i}\mscr{D}_j$ are nonempty, disjoint open subsets of $\mscr{D}$. As a consequence, we may suppose $\mscr{D}\subset\R^n$ and $\widetilde{\mscr{D}}$ are connected, and hence path connected as $\mscr{D}$ is a subset of $\R^n$ endowed with the topology induced by the Euclidean topology of $\R^n$ and $\widetilde{\mscr{D}}$ admits a graph structure $\mathcal{R}(\mscr{D})$. 
	
	Observe that  $\rho:\mscr{D}\to\widetilde{\mscr{D}}$ is proper. Indeed, let $F\subset\widetilde{\mscr{D}}$ be a compact subset. As $\widetilde{\mscr{D}}$ admits a graph structure, $F$ is closed, therefore $\rho^{-1}(F)$ is a closed subset of $\pi^{-1}(\pi(F))\cap\mscr{D}$, which is compact since $\pi|_{\mscr{D}}$ is proper. Hence, $\rho^{-1}(F)$ is compact, which proves that $\rho:\mscr{D}\to\widetilde{\mscr{D}}$ is proper. Recall that $\mscr{D}$ and $\widetilde{\mscr{D}}$ are path connected, or $0$-connected in the terminology of \cite{Sma57}, and locally compact and separable metric spaces as $\mscr{D}$ is a closed submanifold of $\R^n$ endowed with the induced Euclidean topology of $\R^n$ and $\widetilde{\mscr{D}}$ admits a graph structure. Let us show that $\mscr{D}\subset\R^n$ is LC$^n$ and $\rho^{-1}(x)=\pi^{-1}(\pi(x))\cap \mscr{D}$ is LC$^{n-1}$. We recall that a topological space $X$ is LC$^{m}$ (locally $m$-connected) if for every $x\in X$ and for every neighborhood $U$ of $x$ there is an open neighborhood $V\subset U$ of $x$ such that every continuous function $f:\mathbb{S}^k\to V$ is null homotopic in $U$ for every $k\in\{0,\dots,m\}$. Observe that, as $\mscr{D}$ is a manifold with boundary, then it is LC$^{n}$ as the local structure at any point is either a ball $B^{n}$ or $B^{n-1}\times [0,1[$. For the same reason, if $\pi(x)$ is a regular value for $\pi|_{\mscr{D}}$, then $\rho^{-1}(x)$ is a connected component of $\pi^{-1}(\pi(x))\cap\mscr{D}$, which is a manifold with boundary of dimension $n-1$, hence it is LC$^{n-1}$. Let $x\in \mscr{D}$ such that $\pi(x)$ is a critical value for $\pi|_{\mscr{D}}$. As before, if $p\in\rho^{-1}(x)$ is not a critical point of $\pi|_{\partial\mscr{D}}$, then the local model of $\pi^{-1}(\pi(x))\cap\mscr{D}$ at $p$ is either a $B^{n-1}$ or $B^{n-2}\times[0,1[$. Let $p\in\rho^{-1}(x)$ be a critical point of $\pi|_{\mscr{D}}$, then by Lemma \ref{lem:Morse} there is a (closed) neighborhood $U$ of $p$ in $\R^n$ such that $(\pi^{-1}(\pi(p))\cap \mscr{D})\cap U$ is homeomorphic to the cone $\textnormal{Cone}(\mathbb{S}^{k-1}\times\mathbb{D}^{n-k-1})$ of $\mathbb{S}^{k-1}\times\mathbb{D}^{n-k-1}$ at the origin $\overline{0}'$ of $\R^{n-1}$, where $k$ denotes the index of $p$ as a critical point of $\pi|_{\mscr{D}}$. As $\rho^{-1}(x)$ is the connected component of $\pi^{-1}(\pi(p))\cap \mscr{D}$ containing $p$ and $\textnormal{Cone}(\mathbb{S}^{k-1}\times\mathbb{D}^{n-k-1})$ is contractible, we conclude that $\rho^{-1}(x)$ is LC$^{n-1}$.
		
	Moreover, by assumption, for every $x\in\widetilde{\mscr{D}}$ we have that $\pi_k(\rho^{-1}(x))$ is trivial for every $k\in\{0,\dots,n-1\}$, where $\pi_k(*)$ denotes the $k$-th fundamental group operator and $\pi_k(\rho^{-1}(x))$ coincides with a connected component of $\pi^{-1}(\pi(x))\cap \mscr{D}$. In the terminology of \cite{Sma57}, this means that $\pi_k(\rho^{-1}(x))$ is $(n-1)$-connected for every $x\in \widetilde{\mscr{D}}$. Therefore, as $\rho:\mscr{D}\to \widetilde{\mscr{D}}$ is proper, $\mscr{D}$ and $\widetilde{\mscr{D}}$ are path connected, separable metric spaces, $\mscr{\mscr{D}}$ is LC$^n$ and for each $x\in\widetilde{\mscr{D}}$, $\rho^{-1}(x)$ is LC$^{n-1}$ and $(n-1)$-connected, \cite[Main Thm.]{Sma57} ensures that $\rho_*:\pi_k(\mscr{D})\to\pi_k(\widetilde{\mscr{D}})$ is an isomorphism for every $k=0,\dots,n$. As $\rho:\mscr{D}\to\widetilde{\mscr{D}}$ is a continuous function between CW-complexes and $\pi_k(\rho^{-1}(x))$ is trivial for every $k$, then Whitehead's Theorem \cite[Thm.4.5]{Hat02} ensures that $\rho:\mscr{D}\to\widetilde{\mscr{D}}$ is a homotopy equivalence.
\end{proof}

\begin{remark}
	Observe that in Proposition \ref{prop:hom-equiv} it is not necessary to suppose $\mscr{D}\subset\R^n$ is compact as it was required in \cite[Prop.2.24]{BPS24}.
\end{remark}

Observe that the triviality condition of the fundamental groups of each fiber of $\pi|_{\mscr{D}}$ of Proposition \ref{prop:hom-equiv} is not always guaranteed in dimension $n>2$, in contrast with the case $n=2$. Next example shows that, when the condition of Proposition \ref{prop:hom-equiv} is not satisfied, the result does not hold in general.

\begin{example}\label{ex:non-retract}
	Let $\Gamma:=(V,E)$ be a graph defined as $V=\{v_1,v_2,v_3,v_4\}$ and $E=\{[v_1,v_2],[v_2,v_3],$ $[v_3,v_4]\}$ and consider the realization $|\Gamma|\subset\R\times\{0\}\subset\R^n$ as in Figure \ref{fig:ex-homequiv}. Define the compact domain of finite type $\mscr{D}\subset\R^n$ as follows. Let $\partial\mscr{D}:=\mathbb{S}_1\cup \mathbb{S}_2$, where $\mathbb{S}_1$ denotes the unique $(n-1)$-sphere of $\R^n$ containing $v_1$ and $v_4$ whose center is contained in $|\Gamma|$ and $\mathbb{S}_2$ denotes the unique $(n-1)$-sphere of $\R^n$ containing $v_2$ and $v_3$ whose center is contained in $|\Gamma|$. Let $p_i\in\R[x_1,\dots,x_n]$ for $i=1,2$ such that: $\mathcal{Z}_{\R^n}(p_i)=\mathbb{S}_i$, $\{x\in\R^n\,|\,p_i(x)\geq 0\}$ is compact and $\nabla p_i(x)\neq \overline{0}$ for every $x\in S_i$ for $i=1,2$. Define $\mscr{D}:=\{x\in\R^n\,|\, -p_1(x)p_2(x)\geq 0\}$. Observe that, $\mscr{D}$ is a domain of finite type since the polynomials $p_1,p_2\in\R[x_1,\dots,x_n]$ have degree 2, so $\pi|_{\partial\mscr{D}}$ is a Morse function. By construction, $\mscr{D}$ is diffeomorphic to $\mathbb{S}^{n-1}\times[0,1]$, hence it is not contractible even if its Poincaré-Reeb space $\widetilde{\mscr{D}}$ is so as it coincides with the realization $|\Gamma|$ of $\Gamma$.
\end{example}

\begin{figure}
	\tikzset{every picture/.style={line width=0.75pt}} 
	
	\begin{tikzpicture}[x=0.75pt,y=0.75pt,yscale=-0.75,xscale=0.75]
		
			
			\draw [line width=1.5]    (207,135) -- (291,135.75) ;
			\draw [shift={(291,135.75)}, rotate = 0.51] [color={rgb, 255:red, 0; green, 0; blue, 0 }  ][fill={rgb, 255:red, 0; green, 0; blue, 0 }  ][line width=1.5]      (0, 0) circle [x radius= 4.36, y radius= 4.36]   ;
			\draw [shift={(207,135)}, rotate = 0.51] [color={rgb, 255:red, 0; green, 0; blue, 0 }  ][fill={rgb, 255:red, 0; green, 0; blue, 0 }  ][line width=1.5]      (0, 0) circle [x radius= 4.36, y radius= 4.36]   ;
			\draw [line width=1.5]    (375,136.5) -- (459,137.25) ;
			\draw [shift={(459,137.25)}, rotate = 0.51] [color={rgb, 255:red, 0; green, 0; blue, 0 }  ][fill={rgb, 255:red, 0; green, 0; blue, 0 }  ][line width=1.5]      (0, 0) circle [x radius= 4.36, y radius= 4.36]   ;
			\draw [shift={(375,136.5)}, rotate = 0.51] [color={rgb, 255:red, 0; green, 0; blue, 0 }  ][fill={rgb, 255:red, 0; green, 0; blue, 0 }  ][line width=1.5]      (0, 0) circle [x radius= 4.36, y radius= 4.36]   ;
			\draw [line width=1.5]    (291,135.75) -- (375,136.5) ;
			\draw [shift={(375,136.5)}, rotate = 0.51] [color={rgb, 255:red, 0; green, 0; blue, 0 }  ][fill={rgb, 255:red, 0; green, 0; blue, 0 }  ][line width=1.5]      (0, 0) circle [x radius= 4.36, y radius= 4.36]   ;
			\draw [shift={(291,135.75)}, rotate = 0.51] [color={rgb, 255:red, 0; green, 0; blue, 0 }  ][fill={rgb, 255:red, 0; green, 0; blue, 0 }  ][line width=1.5]      (0, 0) circle [x radius= 4.36, y radius= 4.36]   ;
			\draw   (207,135) .. controls (207,86.54) and (263.41,47.25) .. (333,47.25) .. controls (402.59,47.25) and (459,86.54) .. (459,135) .. controls (459,183.46) and (402.59,222.75) .. (333,222.75) .. controls (263.41,222.75) and (207,183.46) .. (207,135) -- cycle ;
			\draw  [draw opacity=0] (450.84,125.52) .. controls (456.11,128.83) and (459,132.42) .. (459,136.17) .. controls (458.99,152.74) and (402.58,166.15) .. (332.99,166.12) .. controls (263.4,166.1) and (207,152.65) .. (207,136.08) .. controls (207,132.18) and (210.14,128.45) .. (215.84,125.03) -- (333,136.13) -- cycle ; \draw   (450.84,125.52) .. controls (456.11,128.83) and (459,132.42) .. (459,136.17) .. controls (458.99,152.74) and (402.58,166.15) .. (332.99,166.12) .. controls (263.4,166.1) and (207,152.65) .. (207,136.08) .. controls (207,132.18) and (210.14,128.45) .. (215.84,125.03) ;  
			\draw  [draw opacity=0][dash pattern={on 4.5pt off 4.5pt}] (221.11,123.37) .. controls (241.67,113.88) and (284.25,107.32) .. (333.46,107.25) .. controls (384.89,107.18) and (429.11,114.22) .. (448.51,124.36) -- (333.5,136.06) -- cycle ; \draw  [dash pattern={on 4.5pt off 4.5pt}] (221.11,123.37) .. controls (241.67,113.88) and (284.25,107.32) .. (333.46,107.25) .. controls (384.89,107.18) and (429.11,114.22) .. (448.51,124.36) ;  
			\draw   (291,135.75) .. controls (291,123.39) and (309.8,113.38) .. (333,113.38) .. controls (356.2,113.38) and (375,123.39) .. (375,135.75) .. controls (375,148.11) and (356.2,158.13) .. (333,158.13) .. controls (309.8,158.13) and (291,148.11) .. (291,135.75) -- cycle ;
			\draw  [draw opacity=0][dash pattern={on 4.5pt off 4.5pt}] (294.45,140.2) .. controls (292.87,139.29) and (292,138.32) .. (292.01,137.31) .. controls (292.02,132.64) and (310.61,128.92) .. (333.53,129) .. controls (356.01,129.08) and (374.3,132.79) .. (374.98,137.33) -- (333.5,137.46) -- cycle ; \draw  [dash pattern={on 4.5pt off 4.5pt}] (294.45,140.2) .. controls (292.87,139.29) and (292,138.32) .. (292.01,137.31) .. controls (292.02,132.64) and (310.61,128.92) .. (333.53,129) .. controls (356.01,129.08) and (374.3,132.79) .. (374.98,137.33) ;  
			\draw  [draw opacity=0] (374.89,137.17) .. controls (373.1,141.86) and (355.2,145.4) .. (333.42,145.22) .. controls (310.48,145.02) and (291.92,140.76) .. (291.97,135.7) .. controls (291.97,135.66) and (291.97,135.63) .. (291.97,135.59) -- (333.5,136.06) -- cycle ; \draw   (374.89,137.17) .. controls (373.1,141.86) and (355.2,145.4) .. (333.42,145.22) .. controls (310.48,145.02) and (291.92,140.76) .. (291.97,135.7) .. controls (291.97,135.66) and (291.97,135.63) .. (291.97,135.59) ;  
			
			\draw (80,123.4) node [anchor=north west][inner sep=0.75pt]    {$\mathcal{R}(\mathscr{D}) =|\Gamma |$};
			\draw (476,43.4) node [anchor=north west][inner sep=0.75pt]    {$\mathscr{D} \subset \mathbb{R}^{n}$};

		\end{tikzpicture}
		
	\caption{A domain of finite type $\mscr{D}\subset\R^n$ which is not homotopically equivalent to its Poincaré-Reeb space $\widetilde{\mscr{D}}$, for $n\geq 3$.}\label{fig:ex-homequiv}
\end{figure}	

\subsection{Topological constraints in realizing Poincaré-Reeb graphs} In this paper we are interested in realizing domains of finite type $\mscr{D}\subset\R^n$ with a prescribed Poincaré-Reeb graph. 
Let us recall a notion of orientation for graphs introduced by Sharko \cite{Sha06} to study the realization of Morse functions on surfaces with prescribed Poincaré-Reeb graphs and then used by Michalak in \cite{Mic18} in higher dimensions. 

\begin{definition}\label{def:orientation}
	We say that a graph $\Gamma:=(V,E)$ (without unbounded edges) admits a \emph{good orientation} if there exists a continuous function $h:|\Gamma|\to\R$ such that $h$ is strictly monotonic on each edge and has extrema on vertices of degree $1$. Such a continuous function $h:|\Gamma|\to\R$ is called an \emph{orientation} of $\Gamma$. A graph $\Gamma$ endowed with an orientation $h$ is called an \emph{oriented graph} and is denoted by $(\Gamma,h)$.
\end{definition}

Observe that each orientation $h:|\Gamma|\to\R$ on a graph $\Gamma:=(V,E)$ induces a partial order $\prec_h$ on the set of vertices: for every $v_1,v_2\in V$ we say that $v_1\prec_h v_2$ if and only if $h(v_1)<h(v_2)$.

\begin{definition}\label{def:equiv-orientations}
	Let $\Gamma:=(V,E)$ be a graph and $h_1,h_2:|\Gamma|\to\R$ two orientations of $\Gamma$. We say that $h_1$ and $h_2$ are \emph{equivalent} if the induced orderings $\prec_{h_1}$ and $\prec_{h_2}$ on $V$ do coincide.
\end{definition}

\begin{definition}\label{def:iso-graph}
	Let $\Gamma:=(V,E)$ and $\Gamma':=(V',E')$ be two graphs endowed with the orientations $h:|\Gamma|\to\R$ and $h':|\Gamma'|\to\R$, respectively. We say that an isomorphism of graphs $\phi:\Gamma\to\Gamma'$ is an \emph{isomorphism of ordered graphs} if the orientations $h:|\Gamma|\to\R$ and $h'\circ\widetilde{\phi}:|\Gamma|\to\R$ on $\Gamma$ are equivalent, where $\widetilde{\phi}:|\Gamma|\to|\Gamma'|$ denotes the continuous map induced by $\phi$ on the realizations of $\Gamma$ and $\Gamma'$. If such an isomorphism of ordered graphs $\phi:\Gamma\to\Gamma'$ exists, we say that the ordered graphs $(\Gamma,h)$ and $(\Gamma',h')$ are \emph{isomorphic} and we write $(\Gamma,h)\cong(\Gamma',h')$.
\end{definition}

The next result shows the existence of a good orientation on the Poincaré-Reeb graph $\mathcal{R}(\mscr{D})$ of a domain of finite type $\mscr{D}\subset\R^n$ induced by $\pi|_{\mscr{D}}$.

\begin{lemma}\label{cor:orientation}
	Let $\mscr{D}\subset\R^n$ be a compact domain of finite type. Then its Poincaré-Reeb graph $\mathcal{R}(\mscr{D})$ admits a good orientation.
\end{lemma}

\begin{proof}
	Consider $\rho:\mscr{D}\to \widetilde{\mscr{D}}$ be the quotient map. By definition of the equivalence relation defining $\widetilde{\mscr{D}}$, whenever $\rho(x)=\rho(y)$ for $x,y\in\mscr{D}$, then $\pi(x)=\pi(y)$, hence the universal property of the quotient topology provides a continuous function $h:\widetilde{\mscr{D}}\to \R$ such that $\pi=h\circ\rho$. Therefore, maxima and minima of $\pi$ correspond to maxima and minima of $h$, and hence they appear just at vertices of the graph structure $\mathcal{R}(\mscr{D})$ of $\widetilde{\mscr{D}}$ defined in Subsection \ref{sub:1.1}. Let $b,c$ be two critical values of $\pi|_{\partial\mscr{D}}$ and assume there is an edge $e$ of $\mathcal{R}(\mscr{D})$ such that $e=(v,w)$ and $\pi(v)=b<c=\pi(w)$, that is, there is a connected component $D$ of $\pi^{-1}(]b,c[])\cap \mscr{D}$ whose closure in $\mscr{D}$ intersects the connected component of $\pi^{-1}(b)\cap\mscr{D}$ and $\pi^{-1}(c)\cap\mscr{D}$ containing $v$ and $w$, respectively, that does not contain any critical point of $\pi|_{\partial\mscr{D}}$. Therefore, $d (\pi|_{\partial\mscr{D}})_x\neq 0$ for every $x\in D$ and hence the function $h$ restricted to the realization $e$ of each edge $e$ of $\mathcal{R}(\mscr{D})$ in $\widetilde{\mscr{D}}$ is strictly monotonic.
\end{proof}

\begin{remark}\label{rem:orientation}
	A similar proof of Lemma \ref{cor:orientation} shows that the Reeb space $\widetilde{\partial\mscr{D}}:=\partial\mscr{D}/\sim$ of the Morse function $\pi|_{\partial\mscr{D}}$, where the equivalence relation $\sim$ on $\partial\mscr{D}$ is defined as $x\sim y$ if $\pi(x)=\pi(y)$, induces a good orientation on the Reeb graph $\mathcal{R}(\pi|_{\partial\mscr{D}})$ of $\pi|_{\partial\mscr{D}}$ as a Morse function. However, the two graphs $\mathcal{R}(\mscr{D})$ and $\mathcal{R}(\pi|_{\partial\mscr{D}})$ are in general not isomorphic. Indeed, Example \ref{ex:non-retract} shows that $\mscr{D}$ and $\partial\mscr{D}$ may have a different number of connected components. 
\end{remark}

Observe that Lemma \ref{cor:orientation} above clarifies why some graphs $\Gamma:=(V,E)$ can not be obtained as a Poincaré-Reeb graph of a domain of finite type $\mscr{D}\subset\R^n$, even when requiring that $\Gamma$ has just vertices of degree $1$ or $3$. See Figure \ref{fig:ex1} for an example of a graph (whose vertices have degree $1$ and $3$) which does not admit any good orientation since it has no vertices of degree $1$. See Figure \ref{fig:ex2} for an example of an embedded graph which can be realized as the Poincaré-Reeb graph of a domain of finite type depending on its realization $|\Gamma|\subset\R^n$.

\begin{figure}[h!]
	\tikzset{every picture/.style={line width=0.75pt}} 
	
	\begin{tikzpicture}[x=0.75pt,y=0.75pt,yscale=-0.75,xscale=0.75]
		
		\draw [line width=1.5]    (253,141.5) -- (400,141.5) -- (409,141.5) ;
		\draw [shift={(409,141.5)}, rotate = 0] [color={rgb, 255:red, 0; green, 0; blue, 0 }  ][fill={rgb, 255:red, 0; green, 0; blue, 0 }  ][line width=1.5]      (0, 0) circle [x radius= 4.36, y radius= 4.36]   ;
		\draw [shift={(253,141.5)}, rotate = 0] [color={rgb, 255:red, 0; green, 0; blue, 0 }  ][fill={rgb, 255:red, 0; green, 0; blue, 0 }  ][line width=1.5]      (0, 0) circle [x radius= 4.36, y radius= 4.36]   ;
		\draw  [line width=1.5]  (253,141.5) .. controls (253,98.42) and (287.92,63.5) .. (331,63.5) .. controls (374.08,63.5) and (409,98.42) .. (409,141.5) .. controls (409,184.58) and (374.08,219.5) .. (331,219.5) .. controls (287.92,219.5) and (253,184.58) .. (253,141.5) -- cycle ;
		
		\draw (430,63.4) node [anchor=north west][inner sep=0.75pt]    {$|\Gamma|$};
	\end{tikzpicture}
	\caption{Example of a graph $\Gamma$ which does not admit any good orientation.}\label{fig:ex1}
	\end{figure}
	\vspace{1em}
	\begin{figure}
	\tikzset{every picture/.style={line width=0.75pt}} 
	
	\begin{tikzpicture}[x=0.75pt,y=0.75pt,yscale=-0.75,xscale=0.75]
		
		\draw  [line width=1.5]  (98,136.75) .. controls (98,102.09) and (126.09,74) .. (160.75,74) .. controls (195.41,74) and (223.5,102.09) .. (223.5,136.75) .. controls (223.5,171.41) and (195.41,199.5) .. (160.75,199.5) .. controls (126.09,199.5) and (98,171.41) .. (98,136.75) -- cycle ;
		\draw [line width=1.5]    (35,136.25) -- (98,136.75) ;
		\draw [shift={(98,136.75)}, rotate = 0.45] [color={rgb, 255:red, 0; green, 0; blue, 0 }  ][fill={rgb, 255:red, 0; green, 0; blue, 0 }  ][line width=1.5]      (0, 0) circle [x radius= 4.36, y radius= 4.36]   ;
		\draw [shift={(35,136.25)}, rotate = 0.45] [color={rgb, 255:red, 0; green, 0; blue, 0 }  ][fill={rgb, 255:red, 0; green, 0; blue, 0 }  ][line width=1.5]      (0, 0) circle [x radius= 4.36, y radius= 4.36]   ;
		\draw [line width=1.5]    (223.5,136.75) -- (286.5,137.25) ;
		\draw [shift={(286.5,137.25)}, rotate = 0.45] [color={rgb, 255:red, 0; green, 0; blue, 0 }  ][fill={rgb, 255:red, 0; green, 0; blue, 0 }  ][line width=1.5]      (0, 0) circle [x radius= 4.36, y radius= 4.36]   ;
		\draw [shift={(223.5,136.75)}, rotate = 0.45] [color={rgb, 255:red, 0; green, 0; blue, 0 }  ][fill={rgb, 255:red, 0; green, 0; blue, 0 }  ][line width=1.5]      (0, 0) circle [x radius= 4.36, y radius= 4.36]   ;
		\draw [line width=1.5]    (471,77) .. controls (491,39) and (556,38) .. (571,77) ;
		\draw [shift={(571,77)}, rotate = 68.96] [color={rgb, 255:red, 0; green, 0; blue, 0 }  ][fill={rgb, 255:red, 0; green, 0; blue, 0 }  ][line width=1.5]      (0, 0) circle [x radius= 4.36, y radius= 4.36]   ;
		\draw [shift={(471,77)}, rotate = 297.76] [color={rgb, 255:red, 0; green, 0; blue, 0 }  ][fill={rgb, 255:red, 0; green, 0; blue, 0 }  ][line width=1.5]      (0, 0) circle [x radius= 4.36, y radius= 4.36]   ;
		\draw [line width=1.5]    (410,184) .. controls (414,204) and (429,256) .. (506,225) ;
		\draw [shift={(506,225)}, rotate = 338.07] [color={rgb, 255:red, 0; green, 0; blue, 0 }  ][fill={rgb, 255:red, 0; green, 0; blue, 0 }  ][line width=1.5]      (0, 0) circle [x radius= 4.36, y radius= 4.36]   ;
		\draw [shift={(410,184)}, rotate = 78.69] [color={rgb, 255:red, 0; green, 0; blue, 0 }  ][fill={rgb, 255:red, 0; green, 0; blue, 0 }  ][line width=1.5]      (0, 0) circle [x radius= 4.36, y radius= 4.36]   ;
		\draw [line width=1.5]    (410,184) .. controls (365,141) and (400,58) .. (471,77) ;
		\draw [line width=1.5]    (410,184) .. controls (478,204) and (520,123) .. (471,77) ;
		
		\draw (258,45.4) node [anchor=north west][inner sep=0.75pt]    {$|\Gamma|_1$};
		\draw (598,44.4) node [anchor=north west][inner sep=0.75pt]    {$|\Gamma|_2$};

	\end{tikzpicture}
	\caption{Example of a graph $\Gamma$ which does admit a good orientation induced by $\pi$ depending on the realization $|\Gamma|_1$ or $|\Gamma|_2$.}\label{fig:ex2}
\end{figure}

\begin{notation}
	From now on we will always consider the Poincaré-Reeb graph $\mathcal{R}(\mscr{D})$ of a domain of finite type $\mscr{D}\subset\R^n$ endowed with the good orientation $h$ induced by $\pi$ as in Lemma \ref{cor:orientation}. We will continue writing $\mathcal{R}(\mscr{D})$ instead of $(\mathcal{R}(\mscr{D}),h)$ since the orientation $h$ of $\mathcal{R}(\mscr{D})$ will be always induced by $\pi$.
\end{notation}

Let us describe how to reinterpret Poincaré-Reeb graphs of domains of finite type as Reeb graphs of Morse functions. For a definition of Reeb graph for a Morse function we refer to \cite[Lem.2.1]{Sha06}.

\begin{proposition}\label{prop:top-inv}
	Let $\mscr{D}\subset\R^n$ be a compact domain of finite type. Then there is a $\cinfty$ hypersurface $M\subset\R^{n+1}$ such that:
	\begin{enumerate}[label=\emph{(\roman*)}, ref=(\roman*)]
		\item $M$ intersects $\R^n\times\{0\}$ transversally and $M\cap(\R^n\times\{0\})=\partial\mscr{D}\times\{0\}$.
		\item $\Pi|_M$ is a Morse function whose Reeb graph $\mathcal{R}(\Pi|_M)$ is isomorphic to $\mathcal{R}(\mscr{D})$, where $\Pi:\R^{n+1}\to\R$ denotes the projection onto the first coordinate.
	\end{enumerate}
\end{proposition}

\begin{proof}
	Let $h\in\cinfty(\R^n)$ such that $\partial\mscr{D}=\mathcal{Z}_{\R^n}(h)$ and $\nabla h(x)\neq \overline{0}$ for every $x\in\partial\mscr{D}$. Then, up to change the sign of $h$, we have that $\mscr{D}=\{x\in\R^n\,|\,h(x)\geq 0\}$. Consider $M\subset\R^{n+1}$ defined as $\{(x,y)\in\R^n\times\R\,|\,y^2=h(x)\}=\mathcal{Z}_{\R^{n+1}}(y^2-h(x))$. Observe that $g(x,y):=y^2-h(x)\in\cinfty(\R^{n+1})$ and $\nabla g(x,y)\neq (\overline{0},0)$ for every $(x,y)\in M$. Indeed, let $(x,y)\in M$, if $h(x)>0$, then $\frac{\partial g}{\partial y}(x,y)=2y\neq 0$, if $h(x)=0$, then $\frac{\partial g}{\partial y}(x,y)=0$ but $\nabla h(x)\neq \overline{0}$, hence in both cases $\nabla g(x,y)\neq (\overline{0},0)$. This shows that $M$ is a compact $\cinfty$ hypersurface of $\R^{n+1}$, as $\partial\mscr{D}$ is a compact $\cinfty$ hypersurface of $\R^{n}$.
	
	Let $\Pi:\R^{n+1}\to\R$ be the projection map onto the first coordinate, that is, $\Pi(x,y)=\pi(x)=x_1$ where $x=(x_1,\dots,x_n)$. Observe that the map $\Pi|_{M}:M\to\R$ is a Morse function for $M$.  Indeed, $(x,y)\in M$ is a critical point of the projection $\Pi|_{M}$ if and only if the only non-null entry of the gradient $\nabla g(x,y)$  is $\frac{\partial g}{\partial x_1}(x,y)$ by Remark \ref{rem:crit}. In particular, since $\frac{\partial g}{\partial y}(x,y)=2y$, we get that each critical point of $\Pi|_{M}$ is at level $y=0$. Observe that, by definition of $g$, we have $\frac{\partial g}{\partial x_i}(x,y)=\frac{\partial h}{\partial x_i}(x)$ for every $i=1,\dots,n$. Hence, we conclude that the critical points of $\Pi|_{M}$ do coincide with the critical points of $\pi|_{\partial\mscr{D}}$ since again a point $x\in\mscr{D}$ is a critical point for $\pi|_{\partial\mscr{D}}$ if and only if the only if the non-null entry of the gradient $\nabla h(x,y)$  is $\frac{\partial h}{\partial x_1}(x)$. Observe that, by definition of $g(x,y)=y^2-h(x)$, each critical point $(x,0)$ is nondegenerate as $x$ is a nondegenerate critical point of $\pi|_{\partial\mscr{D}}$. In addition, for every $z\in\R$ the number of connected components of $\pi^{-1}(z)\cap \mscr{D}$ coincides with the number of connected components of $\Pi^{-1}(z)\cap M$. Indeed, $\Pi^{-1}(z)\cap M=\{(x,y)\in\mscr{D}\times\R\,|\,\pi(x)=z, |y|=\sqrt{h(x)}\}$, that is, it coincides with the union of the graphs of two continuous functions $f_1(x)=\sqrt{h(x)}$ and $f_2(x)=-\sqrt{h(x)}$ on $\mscr{D}\cap\pi^{-1}(z)$ vanishing at $\partial\mscr{D}\cap\pi^{-1}(z)$, therefore if $D$ is a connected component of $\mscr{D}\cap\pi^{-1}(z)$, then $\{(x,y)\in D \times\R\,|\,|y|=\sqrt{h(x)}\}$ is connected. Denote by $D_1,\dots,D_\ell$ the connected components of $\mscr{D}\cap\pi^{-1}(z)$ and denote by $C_i:=\{(x,y)\in D_i \times\R\,|\,|y|=\sqrt{h(x)}\}$ for $i\in\{1,\dots,\ell\}$. By definition of $M$, we have that $M\cap\Pi^{-1}(z)=\bigcup_{i=1}^\ell C_i$. Moreover, each $C_i$ is connected as shown before and $C_i\cap C_j\neq\varnothing$ if and only if $D_i\cap D_j\neq\varnothing$ as $C_i:=\{(x,y)\in D_i \times\R\,|\,|y|=\sqrt{h(x)}\}$, therefore $C_i\cap C_j\neq\varnothing$ for every $i,j\in\{1,\dots,\ell\}$ with $i\neq j$ as $D_1,\dots,D_\ell$ are the connected components of $\mscr{D}\cap\pi^{-1}(z)$. This shows that, the Reeb graph structure $\mathcal{R}(\Pi|_M)$ on $\widetilde{M}$ induced by $\Pi|_{M}$ is isomorphic to the Poincaré-Reeb graph structure $\mathcal{R}(\mscr{D})$ on $\mscr{D}$. 
	
	We show that the Reeb space $\widetilde{M}$ of $\Pi|_{M}$ (as a Morse function) is homeomorphic to $\widetilde{\mscr{D}}$ inducing an isomorphism between the oriented graph structures $\mathcal{R}(\Pi|_M)$ on $\widetilde{M}$ and $\mathcal{R}(\mscr{D})$ on $\widetilde{\mscr{D}}$. Denote by $\Pi_n:\R^{n+1}\to\R^n$ the projection onto the first $n$-coordinates defined as $\Pi_n(x,y)=x$. Observe that $\Pi=\pi\circ\Pi_n$. Consider the following commutative diagram:
	\begin{center}
		\begin{tikzcd} 
			M \arrow{d}{\rho'} \arrow{r}{\Pi_n} & \mscr{D} \arrow{d}{\rho}  \arrow[bend left=20]{l}{\gamma} \arrow{r}{\pi} & \R \\ 
			\widetilde{M} \arrow{r}{\phi} & \widetilde{\mscr{D}} \arrow{ru}{h_{\pi}}
		\end{tikzcd}
	\end{center}
	where $\rho':M\to \widetilde{M}$ is the quotient map and $h_{\pi}:\widetilde{\mscr{D}}\to \R$ is the orientation of Lemma \ref{cor:orientation} induced by $\pi$. By applying the universal property of the quotient topology to $\rho\circ\Pi_n$ we get a unique continuous function $\phi:\widetilde{M}\to \widetilde{\mscr{D}}$ that makes the diagram commute. Moreover, as $\rho\circ\Pi_n$ is surjective, then $\phi\circ\rho'$ is surjective too and, by definition of $\rho':M\to \widetilde{M}$, $\phi$ is also injective. Consider the continuous function $\gamma:\mscr{D}\to M$ defined as $\gamma(x):=(x,\sqrt{h(x)})$. Then, the universal property of the quotient topology applied to $\rho'\circ \gamma$ gives a unique continuous function $\phi':\widetilde{\mscr{D}}\to \widetilde{M}$ such that $\rho'\circ\gamma=\phi'\circ\rho$. Observe that $\phi\circ \phi'=\id_{\widetilde{\mscr{D}}}=\phi'\circ \phi$, therefore $\phi:\widetilde{M}\to \widetilde{\mscr{D}}$ is a homeomorphism and, as shown before, it preserves the graphs structures $\mathcal{R}(\Pi|_M)$ on $\widetilde{M}$ and $\mathcal{R}(\mscr{D})$ on $\mscr{D}$. Finally, the universal property of the quotient topology applied to $\Pi=\pi\circ\Pi_n$ gives that $h_{\pi}\circ\phi:\widetilde{M}\to\R$ is the good orientation on $\mathcal{R}(\Pi|_{M})$ induced by $\Pi|_{M}$ as in Remark \ref{rem:orientation}. Therefore, the two oriented graphs $\mathcal{R}(\mscr{D})$ and $\mathcal{R}(\Pi|_{M})$ are isomorphic, as desired.
\end{proof}

Observe that $M\subset\R^{n+1}$ defined in Proposition \ref{prop:top-inv} is just the compact hypersurface obtained by identifying the boundary of two copies of $\mscr{D}$. That is the reason why we will refer to $M\subset\R^{n+1}$ as the \emph{double} of $\mscr{D}\subset\R^n$. Moreover, as a direct consequence of Proposition \ref{prop:top-inv} and \cite[Thm.5.1]{Gel18}, we get the following upper-bound on the topology of the Poincaré-Reeb graph of a domain of finite type.

\begin{corollary}\label{cor:upper}
	Let $\mscr{D}\subset\R^n$ be a compact domain of finite type and denote by $M\subset\R^{n+1}$ its double. Then,
	\[
	b_1(\widetilde{\mscr{D}})\leq\textnormal{corank}(\pi_1(M)),
	\]
	where $b_1(\widetilde{\mscr{D}})$ denotes the first Betti-number of $\widetilde{\mscr{D}}$ and $\pi_1(M)$ is the fundamental group of $M$. In particular, if $\mscr{D}$ is diffeomorphic to a disk, then $\mathcal{R}(\mscr{D})$ is a tree.
\end{corollary}

\begin{proof}
	By Proposition \ref{prop:top-inv}, we have that $\mathcal{R}(\Pi|_{M})\cong\mathcal{R}(\mscr{D})$, therefore \cite[Thm.5.1]{Gel18} ensures that $b_1(\widetilde{\mscr{D}})=b_1(|\mathcal{R}(\Pi|_{M})|)\leq\textnormal{corank}(\pi_1(M))$. Assume in addition that $\mscr{D}\subset\R^n$ is homeomorphic to a (closed) disk $\mathbb{D}^n$, then the double $M\subset\R^{n+1}$ of $\mscr{D}\subset\R^n$ is homeomorphic to the sphere $\mathbb{S}^{n}$, for $n\geq 2$. Therefore, $\pi_1(M)=0$ and hence an application of the above inequality ensures that $b_1(\widetilde{\mscr{D}})=0$, that is, $\mathcal{R}(\mscr{D})$ is a tree, as $\mscr{D}$ homeomorphic to $\mathbb{D}^n$ implies that $\mscr{D}$, and hence $\widetilde{\mscr{D}}$, are connected as well.
\end{proof}

Observe that in the case $n=2$, the implication of Corollary \ref{cor:upper}: if $\mscr{D}\subset\R^2$ is diffeomorphic to a disk, then $\mathcal{R}(\mscr{D})$ is a tree, follows as a direct consequence of Proposition \ref{prop:hom-equiv}. Indeed, if $n=2$, then the fiber of $\pi|_{\mscr{D}}$ at any point $y\in \R$ is a finite union of intervals, and therefore each connected component of $\pi|_{\mscr{D}}^{-1}(y)$ is contractible. This shows that Proposition \ref{prop:hom-equiv} applies and hence $\mscr{D}\subset\R^2$ is homoropically equivalent to $\widetilde{\mscr{D}}$. As $\mscr{D}\subset\R^2$ is diffeomorphic to a disk, then $\widetilde{\mscr{D}}$ is connected and $b_1(\widetilde{\mscr{D}})=b_1(\mscr{D})=0$, that is, $\mathcal{R}(\mscr{D})$ is a tree. This was the original way of proving the latter property in \cite[Prop.2.24]{BPS24}.

Let us generalize the notion of vertically equivalence introduced in \cite[Def.2.28]{BPS24} for domains of finite type $\mscr{D}\subset\R^n$.

\begin{definition}\label{def:v-equiv}
	Let $\mscr{D},\mscr{D}'\subset\R^n$ be domains of finite type equipped by the orientation induced by the orientation of $\R^n$. We say that $\mscr{D}$ and $\mscr{D}'$ are \emph{vertically equivalent} and we write $\mscr{D}\sim_v \mscr{D}'$ if there are orientation preserving diffeomorphisms $\phi:\mscr{D}\to\mscr{D}'$ and $\psi:\R\to\R$ such that $\psi\circ\pi|_{\mscr{D}}=\pi|_{\mscr{D}'}\circ\phi$.
\end{definition}

The next result shows which topological information is encoded by the Poincaré-Reeb graph $\mathcal{R}(\mscr{D})$ of a domain of finite type $\mscr{D}\subset\R^n$.

\begin{proposition}\label{prop:equiv}
	Let $\mscr{D},\mscr{D}'\subset\R^n$ be domains of finite type. If $\mscr{D}\sim_{v} \mscr{D}'$, then $\mathcal{R}(\mscr{D})\cong \mathcal{R}(\mscr{D}')$.
\end{proposition}

\begin{proof}
	Let $\phi:\mscr{D}\to\mscr{D}'$ and $\psi:\R\to\R$ be diffeomorphisms satisfying Definition \ref{def:v-equiv}. Consider the following commutative diagram:
	\begin{center}
		\begin{tikzcd} 
			\mscr{D}  \arrow[bend right=50, swap]{dd}{\pi|_{\mscr{D}}} \arrow{d}{\rho} \arrow{r}{\phi} & \mscr{D}' \arrow[bend left=50]{dd}{\pi|_{\mscr{D}'}} \arrow{d}{\rho'} \\
			\widetilde{\mscr{D}} \arrow{r}{\widetilde{\phi}} \arrow{d}{h} &  \widetilde{\mscr{D}}' \arrow{d}{h'}\\
			\R  \arrow{r}{\psi} & \R
		\end{tikzcd}
	\end{center}
		where $h:\widetilde{\mscr{D}}\to \R$ and $h':\widetilde{\mscr{D}}'\to \R$ are the good orientations of Lemma \ref{cor:orientation} induced by $\pi|_{\mscr{D}}$ and $\pi|_{\mscr{D}'}$ on $\mathcal{R}(\mscr{D})$ and $\mathcal{R}(\mscr{D}')$, respectively. Let $x,y\in\mscr{D}$ such that $(\rho'\circ\phi)(x)=(\rho'\circ\phi)(y)$. By definition of $\rho'$, we have that $\rho'(\phi(x))=\rho'(\phi(y))$ if and only if $\pi|_{\mscr{D'}}(\phi(x))=\pi|_{\mscr{D}'}(\phi(y))$ and $\phi(x)$ and $\phi(y)$ belong to the same connected component of $\pi^{-1}(\pi(\phi(x)))\cap\mscr{D}'$. As $\pi|_{\mscr{D}'}\circ\phi=\psi\circ\pi|_{\mscr{D}}$, the previous property is satisfied if and only if $\pi(x)=\pi(y)$ and $x$ and $y$ belongs to the same connected component of $\pi^{-1}(\pi(x))\cap\mscr{D}$, that is, if and only if $\rho(x)=\rho(y)$. Therefore, the universal property of the quotient topology applied to $\rho'\circ\phi:\mscr{D}\to\R$ gives a unique continuous function $\widetilde{\phi}:\mscr{D}\to\mscr{D}'$ such that $\widetilde{\phi}\circ\rho=\rho'\circ\phi$. Moreover, $\widetilde{\phi}$ is surjective as $\rho'\circ\phi$ is and $\widetilde{\phi}$ is injective as the $\rho'(\phi(x))=\rho'(\phi(y))$ if and only if $\rho(x)=\rho(y)$. A similar application of the universal property of the quotient topology to $\rho\circ\phi^{-1}:\mscr{D}'\to\widetilde{\mscr{D}}$ provides that $\widetilde{\phi}^{-1}$ is continuous and hence $\widetilde{\phi}$ is a homeomorphism.
	
	Let us show that $\widetilde{\phi}:\widetilde{\mscr{D}}\to\widetilde{\mscr{D}}'$ as above induces an isomorphism between the oriented graphs $\mathcal{R}(\mscr{D})$ and $\mathcal{R}(\mscr{D}')$. Let $v_1,\dots,v_k$ be the critical points of $\pi|_{\partial\mscr{D}}$, then $\phi(v_1),\dots,\phi(v_k)$ are the critical points of $\pi|_{\partial\mscr{D}'}$ as $\phi|_{\partial{\mscr{D}}}:\partial\mscr{D}\to\partial\mscr{D}'$ and $\psi:\R\to \R$ are diffeomorphisms satisfying $\pi\circ\phi|_{\partial\mscr{D}}=\psi\circ\pi$. Therefore, $\widetilde{\phi}$ sends the vertices of $\mathcal{R}(\mscr{D})$ exactly to the vertices of $\mathcal{R}(\mscr{D}')$. Moreover, $\widetilde{\phi}$ preserves the graphs structures $\mathcal{R}(\mscr{D})$ of $\mscr{D}$ and $\mathcal{R}(\mscr{D}')$ of $\mscr{D}'$ as $\pi|_{\mscr{D}'}\circ\phi=\psi\circ\pi|_{\mscr{D}}$. This shows that $\widetilde{\phi}$ induces an isomorphism of graphs between $\mathcal{R}(\mscr{D})$ and $\mathcal{R}(\mscr{D}')$. Let us show that the good orientations $h:\widetilde{\mscr{D}}\to \R$ and $h':\widetilde{\mscr{D}}'\to \R$ are equivalent. Observe that, as $\psi:\R\to\R$ is an orientation preserving diffeomorphism, then $\psi(\pi|_{\mscr{D}}(v_i))<\psi(\pi|_{\mscr{D}}(v_j))$ for $i,j\in\{1,\dots,k\}$ if and only if $\pi|_{\mscr{D}}(v_i)<\pi|_{\mscr{D}}(v_j)$, therefore $h\circ\psi:\widetilde{\mscr{D}}\to \R$ is a good orientation on $\mathcal{R}(\mscr{D})$ equivalent to $h$ according to Definition \ref{def:equiv-orientations}. Observe that, by commutativity of the above diagram, we have that $h\circ\psi=h'\circ\widetilde{\phi}$, showing that the two oriented graphs $\mathcal{R}(\mscr{D})$ and $\mathcal{R}(\mscr{D}')$ are isomorphic.
\end{proof}

We will discuss  in Remark \ref{rem:construction} why the converse of Proposition \ref{prop:equiv} does not hold in general, even in the case of stable domains of finite type we are going to introduce in the next subsection.


\vspace{0.5em}

\subsection{Stable domains of finite type}

For the purpose of this paper we are mostly interested in those domains of finite type $\mscr{D}\subset\R^n$ which are stable by small perturbations, more precisely let us recall the classical notion of stable $\cinfty$ functions.

\begin{definition}\label{def:stable-f}
	Let $M$ be a $\cinfty$ manifold. A function $f\in\cinfty(M)$ is \emph{stable} if there is a neighborhood $\mathcal{U}(f)$ of $f$ in $\cinfty(M)$ such that every $g\in\mathcal{U}(f)$ is conjugate to $f$, i.e. there are diffeomorphisms $\phi:M\to M$ and $\psi:\R\to \R$ such that $\psi\circ f=g\circ\phi$.
\end{definition}

Observe that stable functions are Morse with distinct critical values. Indeed, suppose $f\in\cinfty(M)$ is stable and let $\mathcal{U}(f)$ be a neighborhood of $f$ in $\cinfty(M)$ satisfying Definition \ref{def:stable-f}. Since the set of Morse functions with distinct critical values is open and dense in $\cinfty(M)$, there is such a Morse function $g\in\mathcal{U}(f)$. Since the distinctness of the critical values is invariant in the conjugacy class of $f$ as in Definition \ref{def:stable-f}, we obtain that $f$ is a Morse function with distinct critical values as well. On the other hand, if $M$ is compact, the notions of Morse function with distinct critical values and stable function do coincide, see \cite[Ch.\,III,\,Prop.\,2.2]{GG73}. This is a key property for our purposes.

\begin{definition}
	Let $\mscr{D}\subset\R^n$ be a domain of finite type. We say that $\mscr{D}\subset\R^n$ is \emph{stable} if $\pi|_{\partial\mscr{D}}$ is stable. 
\end{definition}

Observe that when $n=2$ a domain of finite type $\mscr{D}\subset\R^n$ is stable if and only if its Poincaré-Reeb graph $\Gamma(\mscr{D})$ has only vertices of degree $1$ and $3$. The graphs whose vertices have degree 1 or 3 are called \emph{generic} in \cite{BPS24}. This characterization holds just when $n=2$, indeed there are stable domains of finite type $\mscr{D}\subset\R^n$ for $n>2$ whose Poincaré-Reeb graph $\mathcal{R}(\mscr{D})$ has vertices of degree $2$, see Example \ref{ex:non-retract}. However, as a direct consequence of Proposition \ref{prop:top-inv}, we obtain the following description for Poincaré-Reeb graphs of stable domains of finite type $\mscr{D}\subset\R^n$. 

\begin{proposition}
	Let $\mscr{D}\subset\R^n$ be a compact stable domain of finite type. Then its Poincaré-Reeb graph $\mathcal{R}(\mscr{D})$ only has vertices of degree $\leq 3$.
\end{proposition}

\begin{proof}
	Let $M\subset\R^{n+1}$ be the double of $\mscr{D}\subset\R^n$. By Proposition \ref{prop:top-inv} we have that $\mathcal{R}(\mscr{D})$ is isomorphic to $\mathcal{R}(\Pi|_{M})$, where $\Pi|_{M}$ is a stable Morse function as $\pi|_{\partial\mscr{D}}$ is so. Let $z\in\R$ be a critical point for $\Pi|_{M}$ and $v\in M\cap(\R^{n}\times\{0\})=\partial\mscr{D}\times\{0\}$ be the unique critical point with critical value $z$. Let $\delta>0$ be sufficiently small such that $z\in [z-\delta,z+\delta]$ is the unique critical value of $\Pi|_{M}$ and denote by $M^{z'}$ the $n$-manifold with boundary $\Pi|_{M}^{-1}(]-\infty,z'])$ for every regular value $z'\in\R$. Let $k\in\{0,\dots,n\}$ be the index of $v$ as a critical point of $\Pi|_{M}$. As $v\in M^{z+\delta}$ is the unique critical point of $\Pi|_{M}$ in $\Pi|_{M}^{-1}([z-\delta,z+\delta])$, we have that $M^{z+\delta}$ retracts by deformation to $M^{z-\delta}\cup e^k$, where $e^k$ is a $k$-cell contained in $M^{z+\delta}\setminus\Pi^{-1}(z+\delta)$ such that $e^k\cap \Pi|_{M}^{-1}(z-\delta)=\partial e^k$ by \cite[Thm.3.1,\S6]{Hir94}. As the boundary $\partial e^k$ of a $k$-cell $e^k$ is disconnected if and only if $k=1$ and in this case the number of connected components of $\partial e^k$ is $2$, we obtain that $\Pi|_{M}^{-1}(z-\delta)\sqcup\Pi|_{M}^{-1}(z+\delta)$ has at most $3$ connected components in the case of critical points of index $1$ and $n-1$, which is the dual case by applying \cite[Thm.3.2,\S6]{Hir94} to the Morse function $-\Pi|_{M}$. This shows that the degree of $v\in\mathcal{R}(\Pi|_{M})$ is at most $3$, as desired.
\end{proof}

\begin{remark}
	Let $\mscr{D},\mscr{D}'\subset\R^n$ be domains of finite type such that $\mscr{D}\sim_{v}\mscr{D}'$. If $\mscr{D}$ is stable, then $\mscr{D}'$ is so.
\end{remark}

\vspace{0.5em}

\subsection{Algebraic and globally algebraic domains of finite type}\label{sub:1.4}

We introduce a notion of algebraic domain of finite type $\mscr{D}\subset\R^2$ similarly to \cite{BPS24}. However, we are interested in a global version of this notion. We recall that, whenever $X\subset\R^n$ is a real algebraic set, a connected component $C$ of $X$ is said to be \emph{nonsingular} if $C\subset \textnormal{Reg}(X)$, where $\textnormal{Reg}(X)$ denotes the nonsingular locus of $X$, see \cite[Def.3.3.9\& Not.3.3.13]{BCR98}.

\begin{definition}\label{def:alg-dom}
	An \emph{algebraic domain of finite type} $\mscr{D}\subset\R^n$ is a domain of finite type whose boundary $\partial\mscr{D}$ is a union of nonsingular connected components of a real algebraic set. 
\end{definition}

\begin{remark}\label{rmk:Nash}
	Observe that every algebraic domain $\mscr{D}\subset\R^n$ is actually a Nash submanifold with boundary of $\R^n$, we refer to \cite{BCR98,Shi87} for more details about Nash manifolds and functions. We may assume $\mscr{D}$ is connected. Indeed, let $\partial\mscr{D}=\bigsqcup_{i=1}^\ell N_i$ where each $N_i$ is a nonsingular connected component of an algebraic hypersurface, hence a Nash manifold. Let $h_i\in\mathcal{N}(\R^n)\subset\cinfty(\R^n)$ be a Nash function such that $N_i=\mathcal{Z}_{\R^n}(h_i)$ and $\nabla h_i(x)\neq \overline{0}$ for every $x\in N_i$, for every $i=1,\dots,\ell$. As $\mscr{D}$ is a compact connected $\cinfty$ hypersurface it divides $\R^n$ in two connected components in terms of the strictly positivity or negativity of $h_i$, $\mscr{D}\setminus\partial\mscr{D}$ is contained in exactly a connected component of $\R^n\setminus N_i$ for every $i=1,\dots,\ell$. So, up to changing the sign of the $h_i$'s if needed, we obtain that:
	\[
	\mscr{D}=\{(x_1,\dots,x_n)\in\R^n\,|\,h_i(x_1,\dots,x_n)\geq 0,\, i=1,\dots,\ell\}.
	\]
	is a semialgebraic set, hence a Nash submanifold with boundary of $\R^n$ of dimension $n$.
	
	Observe that, just considering products of the $h_i$'s defining $\partial\mscr{D}$, a single Nash inequality is needed in the description of an algebraic domain of finite type $\mscr{D}\subset\R^n$. By contrast, if we wish to describe $\mscr{D}\subset\R^n$ by polynomial inequalities a (finite) system is required in general.
\end{remark}

We are mostly interested in a strictly smaller class of algebraic domains $\mscr{D}\subset\R^n$ we call \emph{globally algebraic domains}, since their boundary is a nonsingular algebraic hypersurface of $\R^n$ and not just a collection of its connected components.

\begin{definition}\label{def:glob-alg-dom}
	We say that a subset $\mscr{D}\subset\R^n$ is a \emph{globally algebraic domain of finite type} if $\mscr{D}\subset\R^n$ is a domain of finite type and there exists a polynomial $g\in\R[x_1,\dots,x_n]$ such that the zero locus $\mathcal{Z}_{\R^n}(g)$ of $g$ in $\R^n$ is a nonsingular algebraic hypersurface of $\R^n$ and
	\[
	\mscr{D}=\{(x_1,\dots,x_n)\in\R^n\,|\,g(x_1,\dots,x_n)\geq 0\}.
	\]
\end{definition}

\begin{example}
	Consider the algebraic domain $\mscr{D}\subset\R^2$ defined as
	\[
	\mscr{D}:=\Big\{(x,y)\in\R^2\,|\,x^3-x-y^2\geq 0,\,\Big(x+\frac{1}{2}\Big)^2+y^2-1\leq 0\Big\}.
	\]
	Observe that $\partial\mscr{D}$ consists exactly of a compact connected component of the elliptic curve given by the equation $y^2=x^3-x$, hence $\mscr{D}$ is a compact algebraic domain. By contrast, $\mscr{D}$ is not a globally algebraic domain since $\partial\mscr{D}$ is just a connected component of an irreducible algebraic set, thus it is not an algebraic set by itself.
\end{example}

\vspace{1em}

\section{Algebraic realization of Poincaré-Reeb graphs}\label{sec:2}

\subsection{Algebraic approximation of stable domains of finite type} In this section we construct stable globally algebraic domains via algebraic approximation. Let us briefly recall some notation originally introduced in the recent papers \cite{FG,GSa,Sav1} in the special case of algebraic hypersurfaces $X\subset\R^n$.

Let $X\subset\R^n$ be an algebraic set. We say that $X\subset\R^n$ is \emph{$\Q$-algebraic} if there exists a polynomial $p\in\Q[x_1,\dots,x_n]$ such that $X=\mathcal{Z}_{\R^n}(p):=\{x\in\R^n\,|\,p(x)=0\}$. The set of $\Q$-algebraic subsets of $\R^n$ constitutes the family of closed sets of a topology of $\R^n$ which is strictly coarser then the usual Zariski topology of $\R^n$. Assume $X$ is a $\Q$-algebraic subset of $\R^n$ and denote by $\Ii_{\Q}(X):=\Ii(X)\cap\Q[x_1,\dots,x_n]$ the ideal of $\Q[x_1,\dots,x_n]$ of those polynomials vanishing at $X\subset\R^n$. Assume $X\subset\R^n$ is an hypersurface. We say that $X\subset\R^n$ is \emph{$\Q$-nonsingular} if there is $q\in\Ii_{\Q}(X)$ such that $\mathcal{Z}_{\R^n}(q)=X$ and $\nabla q(x)\neq \overline{0}$ for every $x\in X$. Clearly, if $X\subset\R^n$ is $\Q$-nonsingular, then it is nonsingular in the usual sense. We refer to \cite{FG,GSa} for further information on $\Q$-algebraic sets and their singularities in the general setting.

\begin{remark}
	There are $\Q$-algebraic hypersurfaces $X\subset\R^n$ that are nonsingular, whereas they are not $\Q$-nonsingular. Consider for instance the $\Q$-algebraic line $X\subset\R^2$:
	\[
	X:=\{(x_1,x_2)\in\R^2\,|\,x_2-\sqrt[3]{2}x_1=0\}=\{(x_1,x_2)\in\R^2\,|\,q(x_1,x_2)=x_2^3-2x_1^3=0\}.
	\]
	It is easy to see that $\Ii_{\Q}(X)=(q)$ but $\nabla q(0,0)=(0,0)$, hence $X\subset\R^n$ is not $\Q$-nonsingular.
\end{remark}

Let $p\in\R[x_1,\dots,x_n]$ be a polynomial and let $p_i\in\R[x_1,\dots,x_n]$ be the unique homogeneous polynomials of degree $i$ such that $p=\sum_{i=0}^dp_i$, where $d$ is the degree of $p$ (as a convention we assume the degree of the zero polynomial to be $0$). We say that  $p\in\R[x_1,\dots,x_n]$ is \emph{overt} if $\mathcal{Z}_{\R^n}(p_d)$ is either empty or $\{0\}$. A $\Q$-algebraic set $V\subset\R^{n}$ is said to be \emph{projectively $\Q$-closed} if there exists an overt polynomial $q\in\Q[x_1,\dots,x_n]$ such that $X=\mathcal{Z}_{\R^n}(q)$. This notion coincides with requiring that $\theta(X)$ is Zariski closed in $\Proy^{n}(\R)$ and it is described by polynomial equations with rational coefficients, where $\theta:\R^{n}\to\Proy^{n}(\R)$ denotes the affine parametrization $\theta(x_1,\ldots,x_{n}):=[1,x_1,\ldots,x_{n}]$. As a consequence, if $X\subset\R^n$ is projectively $\Q$-closed, then it is compact with respect to the usual Euclidean topology of $\R^n$.

\begin{remark}
	There are compact $\Q$-algebraic subsets $X\subset\R^n$ that are not projectively $\Q$-closed. Consider the compact $\Q$-algebraic curve $C:=\{(x_1,x_2)\in\R^2\,|\,x_1^2+x_2^4=1\}$ in $\R^2$. Then, its projective closure in $\Proy^2(\R)$ is given by $C':=\{[x_0,x_1,x_2]\in\Proy^2(\R)\,|\,x_0^4-x_1^2 x_0^2-x_2^4=0\}$ and $[0,1,0]\in C'\setminus\theta(C)$.
\end{remark}

We say that a function $f:\R^n\to\R$ is \emph{$\Q$-regular} if there are polynomials $p,q\in\Q[x_1,\dots,x_n]$ such that $\mathcal{Z}_{\R^n}(q)=\varnothing$ and $f(x)=\frac{p(x)}{q(x)}$ for every $x\in\R^n$.

Let us recall the following preliminary facts on $\cinfty$-approximation.

\begin{remark}\label{rem:implicite}
	Let $M\subset\R^n$ be a $\cinfty$ hypersurface such that $\pi|_M$ is a Morse function. Let $h\in\cinfty(\R^n)$ such that $\mathcal{Z}_{\R^n}(h)=M$ and $\nabla h(x)\neq \overline{0}$ for every $x\in M$. Then:
	\begin{enumerate}
		\item  $x\in M$ is a critical point of $\pi|_M$ if and only if $\frac{\partial h}{\partial x_i}(x)=0$ for every $i=2,\dots,n$.
		\item\label{rem:2} Let $x=(x_1,x')\in M$ be a critical point of $\pi|_M$. Then, by the implicit function theorem, there are a neighborhood $U\subset\R^{n-1}$ of $x'$ and a real valued function $\varphi\in\cinfty(U)$ such that $h(\varphi(y'),y')=0$ for every $y'\in U$. In particular, $\pi|_M(\varphi(y'),y')=\varphi(y')$ for every $y'\in U$, hence:
		\begin{enumerate}[label=(\alph*), ref=(\alph*)]
			\item\label{rem:2a} the first order derivatives of $\pi|_M\circ(\varphi\times\id_{\R^{n-1}}):U\to\R$ at $x'$ are completely determined by the first order derivatives of $h$.
			\item\label{rem:2b} By \ref{rem:2a}, also the Hessian of $\pi|_M\circ(\varphi\times\id_{\R^{n-1}}):U\to\R$ at $x'$ is completely determined by the derivatives of $h$ up to order $2$.
		\end{enumerate}  
	\end{enumerate}
\end{remark}

Our main approximation result reads as follows.

\begin{theorem}\label{thm:approx}
	Every stable compact domain of finite type $\mscr{D}\subset\R^n$ is vertically equivalent to a globally algebraic stable domain of finite type $\mscr{D}'\subset\R^n$. In addition, we may assume that $\partial\mscr{D}'\subset\R^n$ is a projectively $\Q$-closed $\Q$-nonsingular $\Q$-algebraic hypersurface arbitrarily $\cinfty$ close to $\partial\mscr{D}$, that is:
	\begin{enumerate}[label=\emph{(\roman*)}, ref=(\roman*)]
		\item\label{en:thmalg1} there exists $q\in\Q[x_1,\dots,x_n]$ such that $\mscr{D}'=\{x\in\R^n\,|\,q(x)\geq 0\}$, $\partial\mscr{D}'=\mathcal{Z}_{\R^n}(q)$ and $\nabla q(x)\neq \overline{0}$ for every $x\in\partial\mscr{D}'$.
		\item\label{en:thmalg2} $\partial\mscr{D}'\subset\R^n$ coincides with its projective closure in $\mathbb{P}^n(\R)$.
	\end{enumerate}
\end{theorem}

\begin{proof}
	Fix $\mscr{U}\subset\cinfty(\R^n,\R^n)$ a neighborhood of the identity and let $\mscr{U}|_{\partial\mscr{D}}:=\{f\in\cinfty(\partial \mscr{D},\R^n)\,|\,\\f=F|_{\partial\mscr{D}}\,\text{for}\,F\in\mscr{U}\}$ be a neighborhood of the inclusion map $\iota_{\partial\mscr{D}}:\partial \mscr{D}\hookrightarrow\R^n$ in $\cinfty(\partial \mscr{D},\R^n)$. Let $h\in\cinfty(\R^n)$ such that $\mathcal{Z}_{\R^n}(h)=\partial\mscr{D}$, $\mscr{D}:=\{x\in\R^n\,|\,h(x)\geq 0\}$ and $\nabla h(x)\neq \overline{0}$ for every $x\in \partial\mscr{D}$. Fix a compact neighborhood $K$ of $\partial\mscr{D}$ in $\R^n$. Then, by a partition of unity argument, we may suppose $h|_{\R^n\setminus K}=q|_{\R^n\setminus K}$, for some overt polynomial $q\in\Q[x_1,\dots,x_n]$ with $\mathcal{Z}_{\R^n}(q)=\varnothing$ and $q(x)\leq -2$ for every $x\in\R^n\setminus K$. Then, an application of \cite[Lem.3.6]{GSa} to $h-q$ gives a $\Q$-regular function $u:\R^n\to\R$ such that:
	\begin{enumerate}[label=(\alph*), ref=(\alph*)]
		\item\label{en:approx1} There are $e\in\N$ and a polynomial $p'\in\Q[x_1,\dots,x_n]$ of degree $\leq 2e$ such that $u(x):=\frac{p'(x)}{(1+|x|_n^2)^e}$ for every $x\in\R^n$.
		\item\label{en:approx2} $\sup_{x\in\R^n} |h(x)-q(x)-u(x)|< 1$.
		\item\label{en:approx3}$u$ is arbitrarily $\cinfty$ close to $h-q$.
	\end{enumerate}
	Define the $\Q$-regular function $h':\R^n\to\R$ as $h'(x):=u(x)+q(x)$. Observe that $h'(x)=(1+|x|_n^2)^{-e}q'(x)$, where $q'\in\Q[x_1,\dots,x_n]$ is the polynomial $q'(x):=q(x)(1+|x|_n^2)^e+p'(x)$. By \ref{en:approx1}, we get that $q'(x)$ is an overt polynomial. In addition, by \ref{en:approx2} we have that $h'(x)\leq -1$ for every $x\in\R^n\setminus K$, that is, $\mathcal{Z}_{\R^n}(h')=\mathcal{Z}_{\R^n}(q')\subset K\setminus\partial K$ and $\mscr{D}':=\{x\in\R^n\,|\,q'(x)\geq 0\}\subset K\setminus\partial K$. By \ref{en:approx3}, $h'$ is arbitrarily $\cinfty$ close to $h$ and therefore $\nabla h'(x)\neq \overline{0}$ for every $x\in\mathcal{Z}_{\R^n}(h')$. Write $h'(x)=\frac{q_1(x)}{q_2(x)}$ where $q_1,q_2\in\Q[x_1,\dots,x_n]$ with $q_2(x)\neq 0$ for every $x\in\R^n$. Hence, $\nabla q_1(x)=q_2(x)\nabla h'(x)\neq\overline{0}$ for every $x\in \partial\mscr{D}'=\mathcal{Z}_{\R^n}(q_1)$ and therefore $\partial\mscr{D}'\subset\R^n$ is a projectively $\Q$-closed $\Q$-nonsingular $\Q$-algebraic set. Moreover, \ref{en:approx3} and Thom isotopy lemma \cite[Thm.14.1.1]{BCR98} ensure that $\partial\mscr{D}':=\mathcal{Z}_{\R^n}(q')\subset\R^n$ is isotopic to $\partial\mscr{D}$. Denote by $F:\partial\mscr{D}\times[0,1]\to\R^n$ such an isotopy such that $F_0:=\iota_{\partial\mscr{D}}$ and $\Phi:=F_1$ satisfying $\Phi(\partial\mscr{D})=\partial\mscr{D}'$ and $\Phi$ is $\cinfty$ close to $\iota_{\partial\mscr{D}}$. In particular, the embedding $\Phi:\partial\mscr{D}\to\R^n$ can be chosen in $\mscr{U}|_{\partial\mscr{D}}$. Hence, we proved that $\partial\mscr{D}'\subset\R^n$ satisfies \ref{en:thmalg1}\&\ref{en:thmalg2}. We are only left to prove that $\mscr{D}$ and $\mscr{D}'$ are vertically equivalent.
	
	Observe that, since $\Phi$ can be chosen arbitrarily $\cinfty$ close to $\iota_{\partial\mscr{D}}$, $\pi|_{\partial\mscr{D}'}\circ\Phi$ is arbitrarily $\cinfty$ close to $\pi|_{\partial\mscr{D}}$. Thus, as $\mscr{D}$ is stable, up to shrinking the neighborhood $\mscr{U}$ of $\id_{\R^n}$ in $\cinfty(\R^n,\R^n)$ if necessary, we may suppose that $\pi|_{\partial\mscr{D}'}\circ\Phi$ is conjugated to $\pi|_{\partial\mscr{D}}$ in $\cinfty(\partial\mscr{D})$ by \ref{en:approx3} and \cite[Thm.3,\S 3]{Mat69}, that is, there are diffeomorphisms $\phi:\partial\mscr{D}\to \partial\mscr{D}$ and $\psi:\R\to\R$ such that
	\begin{equation}\label{eq:stable}
		\psi\circ\pi|_{\partial\mscr{D}}=(\pi|_{\partial\mscr{D}'}\circ\Phi)\circ\phi=\pi|_{\partial\mscr{D}'}\circ(\Phi\circ\phi).
	\end{equation}
	Since $\partial \mscr{D}\subset\R^n$ is a compact $\cinfty$ hypersurface, it is orientable. Let us show that $\phi:\partial\mscr{D}\to \partial\mscr{D}$ and $\psi:\R\to\R$ are orientation preserving. By \cite[Thm.3,\S 3]{Mat69}, there are a neighborhood $\mscr{V}$ of $\pi|_{\partial\mscr{D}}$ in $\cinfty(\partial\mscr{D})$, continuous functions $H_1:\mscr{V}\to\textnormal{Diff}(\partial\mscr{D})$  and $H_2:\mscr{V}\to\textnormal{Diff}(\R)$ such that $H_1(\pi|_{\partial\mscr{D}})=\id_{\partial\mscr{D}}$ and $H_1(\pi|_{\partial\mscr{D}'}\circ\Phi)=\phi$, $H_2(\pi|_{\partial\mscr{D}})=\id_{\R}$ and $H_2(\pi|_{\partial\mscr{D}'}\circ\Phi)=\psi$. Up to restricting $\mscr{V}$ to its connected component containing $\pi|_{\partial\mscr{D}}$, we may suppose $\mscr{V}$ is connected. Therefore, by continuity of $H_1$ and $H_2$, $H_1(\mscr{V})$ and $H_2(\mscr{V})$ are connected subsets of $\textnormal{Diff}(\partial\mscr{D})$ and of $\textnormal{Diff}(\R)$ containing $\id_{\partial\mscr{D}}$ and $\id_{\R}$, respectively. This shows that the diffeomorphisms $\phi\in H_1(\mscr{V})$ and $\psi\in H_2(\mscr{V})$ are orientation preserving. 
	
	Consider the isotopy $F':\partial\mscr{D}\times[0,1]\to\R^n$ defined as $F':=F\circ(\phi\times\id_\R)$. Observe that, up to shrinking the neighborhood $\mscr{U}$ of $\id_{\R^n}$ in $\cinfty(\R^n,\R^n)$, and therefore the neighborhood $\mscr{U}|_{\partial\mscr{D}}$ of $\iota_{\partial\mscr{D}}$ in $\cinfty(\partial\mscr{D},\R^n)$, we have:
	\begin{enumerate}
		\item\label{thm:approx-1} $F':\partial\mscr{D}\times[0,1]\to\R^n$ is an isotopy arbitrarily close to $\jmath_{\partial\mscr{D}}:=\iota_{\partial\mscr{D}}\circ\phi$.
		\item\label{thm:approx-2} Let $\Phi':=F'_1:\partial\mscr{D}\to\R^n$. If $v_1,\dots,v_k$ are the critical points of $\pi|_{\partial\mscr{D}}$ such that $\pi|_{\partial\mscr{D}}(v_i)<\pi|_{\partial\mscr{D}}(v_{i+1})$ for every $i\in\{1,\dots,k-1\}$, then $\Phi'(v_1),\dots,\Phi'(v_k)$ are the critical points of $\pi|_{\partial\mscr{D}'}$ and, by (\ref{eq:stable}), $\pi|_{\partial\mscr{D}'}(\Phi'(v_i))=\psi(\pi_{\partial\mscr{D}}(v_i))<\psi(\pi_{\partial\mscr{D}}(v_{i+1}))=\pi|_{\partial\mscr{D}}(\Phi'(v_{i+1}))$ for every $i\in\{1,\dots,k-1\}$, as $\psi\in\text{Diff}(\R)$ is orientation preserving. In addition, by Remark \ref{rem:implicite}\ref{rem:2b}, also the index of each critical point is preserved by \ref{en:approx3}.
	\end{enumerate}
	In particular, $\mscr{D}'\subset\R^n$ is a stable domain of finite type. Hence, we are only left to prove that, by choosing $\partial\mscr{D}'=\mathcal{Z}_{\R^n}(q')$ sufficiently $\cinfty$ close to $\partial\mscr{D}$, then $\mscr{D}':=\{x\in\R^n\,|\,q'(x)\geq 0\}$ is vertically equivalent to $\mscr{D}$ (see Definition \ref{def:v-equiv}).

	Observe that, by \cite[Thm.1.3, \S 8]{Hir94}, we may extend the isotopy $F':\partial\mscr{D}\times[0,1]\to\R^n$ to a diffeotopy with compact support $G:\R^n\times[0,1]\to\R^n$ such that $G_1\in\mscr{U}$. Denote by $\partial\mscr{D}_1,\dots,\partial\mscr{D}_\ell$ the connected components of $\partial\mscr{D}$. Let $h_i\in\cinfty(\R^n)$ be such that $\partial\mscr{D}_i=\mathcal{Z}_{\R^n}(h_i)$, $\nabla h_i(x)\neq\overline{0}$ for every $x\in \partial\mscr{D}_i$ and let $\mscr{D}_i:=\{x\in\R^n\,|\,h_i(x)\geq 0\}$ be compact for every $i\in\{1,\dots,\ell\}$. As $G_1$ is a diffeomorphism of $\R^n$ such that $G_1(\partial\mscr{D})=\partial\mscr{D}'$, then $G_1(\partial\mscr{D}_1),\dots,G_1(\partial\mscr{D}_\ell)$ are the connected components of $\partial\mscr{D}'$ and $G_1(\mscr{D}_i)$ is compact, therefore it coincides with the connected compact $\cinfty$ submanifold of $\R^n$ whose boundary is $G_1(\partial\mscr{D}_i)$. Let $h'_i\in\cinfty(\R^n)$ be such that $G_1(\partial\mscr{D}_i)=\mathcal{Z}_{\R^n}(h'_i)$, $\nabla h'_i(x)\neq \overline{0}$ for every $x\in G_1(\partial\mscr{D}_i)$ and $G_1(\mscr{D}_i)=\{x\in\R^n\,|\,h'_i(x)\geq 0\}$. Then, up eventually to a sign, we have that $\mscr{D}$ and $\mscr{D}'$ coincide with the positivity locus of the product of the $h_i$'s and of the $h'_i$'s, respectively. That is, $\mscr{D}=\big\{x\in\R^n\,|\,\pm\prod_{i=1}^{\ell} h_i\geq 0\big\}$ and $\mscr{D}'=\big\{x\in\R^n\,|\,\pm\prod_{i=1}^{\ell} h'_i\geq 0\big\}$, and therefore $G_1(\mscr{D})=\mscr{D'}$ as $\sign(h_i(x))=\sign (h'_i(G_1(x)))$ for every $x\in\R^n$ and $i\in\{1,\dots,\ell\}$. As before, up to shrinking the neighborhood $\mscr{U}$ of $\id_{\R^n}$ in $\cinfty(\R^n,\R^n)$, we have that $\pi|_{\mscr{D}'}\circ G_1|_{\mscr{D}}\in\cinfty(\mscr{D})$ is arbitrarily close to $\pi|_{\mscr{D}}$. We are left to prove that $\pi|_{\mscr{D}'}\circ G_1|_{\mscr{D}}$ is conjugated to $\pi|_{\mscr{D}}$, more precisely, that there are orientation preserving diffeomorphisms $\phi':\mscr{D}\to \mscr{D}$ and $\psi':\R\to\R$ such that
	\begin{equation}\label{eq:stable-2}
		\psi'\circ\pi|_{\mscr{D}}=(\pi|_{\mscr{D}'}\circ G_1|_{\mscr{D}})\circ\phi'=\pi|_{\partial\mscr{D}'}\circ(G_1|_{\mscr{D}}\circ\phi').
	\end{equation}
	
	Observe that by \cite[Thm.2, \S 3]{Mat69}, and its generalization for $\cinfty$ maps between $\cinfty$ manifolds with boundary \cite[Thm.2.3]{vEss78}, it suffices to verify the following claim. In the second reference the author refers to the infinitesimal stability condition below as ``stratum infinitesimal stability".
	
	\vspace{0.5em}
	
	\noindent\textsc{CLAIM (Infinitesimal Stability):} For every $g\in\cinfty(\mscr{D})$ there exist a $\cinfty$ vector field $s:\mscr{D}\to T\mscr{D}$ tangent to $\partial \mscr{D}$ and a $\cinfty$ function $t:\R\to\R$ such that:
	\[
	g=s[\pi|_{\mscr{D}}]+t\circ \pi|_{\mscr{D}},
	\]
	where $s[\pi|_{\mscr{D}}]:\mscr{D}\to\R$ denotes the directional derivative of $\pi|_{\mscr{D}}$ in the direction $s(x)\in T_x \mscr{D}$ for every $x\in \mscr{D}$.
	
	Recall that, by assumption, the Morse function $\pi|_{\partial\mscr{D}}$ has a finite number of critical points and their critical values are distinct, thus choose a $\cinfty$ function $t:\R\to\R$ such that $t\circ\pi|_{\partial\mscr{D}}(v_i)=g(v_i)$ for every critical point $v_i$ of $\pi|_{\partial\mscr{D}}$. Observe that $g$ can be defined for instance by Lagrange interpolation. Then, up to substitute $g$ by $g':=g-t\circ \pi|_\mscr{D}$, it is sufficient to solve the equation $g=s[\pi|_\mscr{D}]$ for $g\in\cinfty(\mscr{D})$ such that $g(v_i)=0$ for every critical point $v_i$ of $\pi|_{\partial\mscr{D}}$. 
	
	Let us define the vector field $s:\mscr{D}\to T\mscr{D}$ locally at $p\in\mscr{D}$ as follows:
	\begin{enumerate}
		\item if $p\in\mscr{D}\setminus\partial\mscr{D}$, choose an open ball $B^n(p,r)\subset\mscr{D}$ centered at $p$ of radius $r>0$ sufficiently small such that $(d\pi|_{B^n(p,r)})_q\neq 0$ for every $q\in B^n(p,r)$. Choose a vector field $s^p:B^n(p,r)\to \R^n$ such that $(d\pi)(s^p)\neq 0$ on $B^n(p,r)$.
		\item if $p\in\partial\mscr{D}$ is a regular point of $\pi|_\mscr{D}$, choose an open ball $B^n(p,r)\subset\R^n$ of radius $r>0$ sufficiently small such that $(d\pi|_{\partial\mscr{D}\cap{B^n(p,r)}})_q\neq 0$ for every $q\in \partial\mscr{D}\cap B^n(p,r)$. Choose a vector field $s^p:\mscr{D}\cap B^n(p,r)\to \R^n$ tangent to $\partial\mscr{D}\cap B^n(p,r)$ such that $(d\pi)(s^p)\neq 0$ on $\mscr{D}\cap B^n(p,r)$.
		\item if $p\in\partial\mscr{D}$ is a critical point of $\pi|_{\partial\mscr{D}}$, by Lemma \ref{lem:Morse} there exists $r>0$, unique $\epsilon_i\in\{\pm 1\}$ for $i\in{2,\dots,n}$ depending on the index of the critical point $p$ of $\pi|_{\mscr{D}}$ and local coordinates $(y_1,\dots,y_{n})$ on $B^n(p,r)\cap\mscr{D}$ such that
		\[
		\pi|_{\mscr{D}\cap B^n(p,r)}(y_1,\dots,y_n)=\pi(p)+y_1+\sum_{i=2}^{n} \epsilon_i y_i^2.
		\]
	\end{enumerate}
	
	As $\mscr{D}\subset\R^n$ is compact, there are $p_1,\dots,p_k$ such that $\mscr{D}\subset\bigcup_{i=1}^k B^n(p_i,r_i)$. Let $\{\rho_1,\dots,\rho_k\}$ be a partition of unity subordinated to the open cover $\mathcal{U}=\{B^n(p_i,r_i)\}_{i=1}^k$. Let us locally define the desired vector field $s_i$ over $B^n(p_i,r_i)$ as follows:
	
	\begin{enumerate}
		\item\label{en:vector-field_1} if $p_i\in\mscr{D}\setminus\partial\mscr{D}$ or $p_i\in\partial\mscr{D}$ is not a critical point of $\pi|_{\partial\mscr{D}}$, define
		\[
		s_i(x):=
		\begin{cases}
			\frac{g(x)\rho_i(x)s^{p_i}(x)}{s^{p_i}[\pi|_{\mscr{D}}](x)} \quad\quad&\text{for $x\in B^n(p_i,r_i)\cap\mscr{D}$},\\
			0  &\text{otherwise.}
		\end{cases}
		\]
		\item\label{en:vector-field_2} if $p_i\in\partial\mscr{D}$ is a critical point of $\pi|_{\partial\mscr{D}}$, then $\rho_i g=\sum_{i=1}^{n} y_i h_i$ in the above local coordinates (since $g(p_i)=0$), for some $h_1,\dots,h_n\in\cinfty(\mscr{D}\cap B^n(p_i,r_i))$. In addition, as $\rho_i$ has compact support, each $h_i$ has compact support as well. Then define
		\[
		s_i:=
		\begin{cases}
			h_1 y_1 \frac{\partial}{\partial y_1} + \sum_{j=2}^{n} \frac{\epsilon_j h_j}{2} \frac{\partial}{\partial y_j}\quad\quad&\text{on $B^n(p_i,r_i)$,}\\
			0  &\text{otherwise.}
		\end{cases}
		\]
	\end{enumerate}
	
	Define the vector field $s(x):=\sum_{i=1}^k s_i(x)$. Then, $s[\pi|_{\mscr{D}}](x)=\sum_{i=1}^k s_i[\pi|_{\mscr{D}}](x)$ where $s_i[\pi|_{\mscr{D}}]=\rho_i g\in\cinfty(\mscr{D})$ for every $i$ such that $p_i$ is not a critical point of $\pi|_{\partial\mscr{D}}$ by (\ref{en:vector-field_1}). By contrast, if $p_i\in\partial\mscr{D}$ is a critical point of $\pi|_{\partial\mscr{D}}$, then
	\begin{align*}
		s_i[\pi|_\mscr{D}]:=&
		\begin{cases}
			\big(h_1 y_1 \frac{\partial}{\partial y_1} + \sum_{j=2}^{n} \frac{\epsilon_j h_j}{2} \frac{\partial}{\partial y_j}\big)(y_1+\sum_{i=2}^{n} \epsilon_i y_i^2) \quad\quad&\text{on $B^n(p_i,r_i)$}\\
			0  &\text{otherwise.}
		\end{cases}\\
		=&
		\begin{cases}
			h_1 y_1 + \sum_{j=2}^{n} \epsilon_j^2 h_j y_j \quad\quad&\text{on $B^n(p_i,r_i)$}\\
			0  &\text{otherwise.}
		\end{cases}
		=
		\begin{cases}
			\sum_{j=1}^{n} h_j y_j \quad\quad&\text{on $B^n(p_i,r_i)$}\\
			0  &\text{otherwise.}
		\end{cases}\\
		=&\rho_i g.
	\end{align*}
	Hence, we got $s[\pi|_{\mscr{D}}]=\sum_{i=1}^k s_i[\pi|_{\mscr{D}}]=\sum_{i=1}^{k}\rho_i g=g$, as desired. $\hfill\blacklozenge$
	
	\vspace{0.5em}
	Finally, by \cite[Thm.2.3]{vEss78}, there are a neighborhood $\mscr{V}'$ of $\pi|_{\mscr{D}}$ in $\cinfty(\mscr{D})$, $\cinfty$ diffeomorphisms $\phi':\mscr{D}\to\mscr{D}$ and $\psi':\R\to\R$ such that $\psi'\circ\pi|_{\mscr{D}}=\pi|_{\partial\mscr{D}'}\circ(G_1|_{\mscr{D}}\circ\phi')$ and continuous functions $H'_1:\mscr{V'}\to\textnormal{Diff}(\partial\mscr{D})$ and $H'_2:\mscr{V'}\to\textnormal{Diff}(\R)$ such that $H'_1(\pi_{\mscr{D}})=\id_{\partial\mscr{D}}$ and $H'_1(\pi|_{\partial\mscr{D}'}\circ\Phi)=\phi'$, $H'_2(\iota_{\partial\mscr{D}})=\id_{\R}$ and $H'_2(\pi|_{\partial\mscr{D}'}\circ\Phi)=\psi'$. Up to restricting $\mscr{V}'$ to its connected component containing $\pi|_{\mscr{D}}$, we may suppose $\mscr{V}'$ is connected. Therefore, $H'_1(\mscr{V'})$ and $H'_2(\mscr{V'})$ are connected by continuity of $H'_1$ and $H'_2$, and then $H'_1(\mscr{V'})$ and $H'_2(\mscr{V'})$ are contained in the same connected components of the identities $\id_{\mscr{D}}$ in $\textnormal{Diff}(\mscr{D})$ and $\id_{\R}$ in $\textnormal{Diff}(\R)$, respectively. This shows that the diffeomorphisms $\phi\in H'_1(\mscr{V'})$ and $\psi\in H'_2(\mscr{V'})$ are orientation preserving. We conclude that $\mscr{D}\subset\R^n$ and $\mscr{D}'\subset\R^n$ are vertically equivalent, as desired.
\end{proof}

\vspace{0.5em}

\subsection{Algebraic realization of domains of finite type with given Poincaré-Reeb graph}

In this section we prove our realization results both for compact globally algebraic domains of finite type. We outline that our Theorem \ref{thm:PR-1} corrects the statement and ameliorate the main theorem of \cite{BPS24} even for $n=2$. Moreover, in Theorems \ref{thm:PR-1}\&\ref{thm:Poincaré-Reeb2} we extend the picture to compact stable domains of finite type in $\R^n$, with $n\geq 3$. We are now in position to state our main realization results.


\begin{theorem}\label{thm:PR-1}
	Let $\Gamma:=(V,E)$ be a finite graph (without unbounded edges) whose vertices have degree $1$ or $3$ with $|\Gamma|\subset\R^n$ such that $\pi|_{|\Gamma|}:|\Gamma|\to \R$ is a good orientation with $\pi(v)\neq\pi(w)$ for every $v,w\in V$ with $v \neq w$. Then, there exists a globally algebraic domain $\mscr{D}\subset\R^n$ whose Poincaré-Reeb graph $\mathcal{R}(\mscr{D})$ is isomorphic to $\Gamma$. In addition:
	\begin{enumerate}[label=\emph{(\roman*)}, ref=(\roman*)]
		\item\label{PR1-1} $\partial\mscr{D}$ can be chosen to be a projectively $\Q$-closed $\Q$-nonsingular $\Q$-algebraic hypersurface of $\R^n$.
		\item\label{PR1-2} $\mscr{D}$ is homotopically equivalent to the Poincaré-Reeb space $\widetilde{\mscr{D}}$.
	\end{enumerate}
\end{theorem}

\begin{proof}
	By \cite[Thm.4.2]{Mic18}, every graph $\Gamma:=(V,E)$ admitting a good orientation whose vertices have degree $1$ or $3$ can be realized as the Reeb graph of a Morse function. More in detail, the construction in \cite{Mic18} provides a compact $\cinfty$ orientable manifold $M$ with a Morse function $f\in\cinfty(M)$ such that:
	\begin{enumerate}[label=(\alph*), ref=(\alph*)]
		\item\label{en:a} each connected component of $f^{-1}(z)$ of each regular value $z\in\R$ is a sphere $\mathbb{S}^{n-1}$,
		\item\label{en:b} the building blocks defined in \cite[Lem.4.1]{Mic18}, in the case of vertices of degree $1$ and $3$, are obtained just by adding a unique $1$ or $(n-1)$-handle to a disjoint union of spheres $\mathbb{S}^{n-1}$ for vertices of degree $3$ and a unique $0$ or $n$-handle for vertices of degree $1$. 
	\end{enumerate}
	As $\pi|_{|\Gamma|}$ is a good orientation on $\Gamma$, each edge $e\in E$ can be realized in $|\Gamma|\subset\R^n$ as a $\cinfty$ curve transverse to the fibers of $\pi$ at any point. Let us outline how the compact $\cinfty$ orientable manifold $M$ as above can be embedded as a $\cinfty$ hypersurface $M\subset\R^{n+1}$ in such a way that:
	\begin{enumerate}
	\item\label{en:1} $M$ is transverse to $\R^n\times\{0\}$.
	\item\label{en:2} The Morse function $f\in\cinfty(M)$ coincides with $\Pi|_M$, where $\Pi:\R^{n+1}\to\R$ denotes the projection onto the first coordinate, and the set of critical points of $\Pi|_{M}$ coincides with $V\subset\R^n\times\{0\}$. This ensures that $\Pi|_{M}$ is a stable Morse function.
	\item\label{en:3} $|\Gamma|\times\{0\}\subset\R^n\times\{0\}$ is contained in the bounded connected component of $\R^{n}\setminus M$.
	\end{enumerate}
	The argument relies on the fact that any elementary cobordism between disjoint spheres $\mathbb{S}^{n-1}$ as in \ref{en:b} can be realized in $\R^{n+1}$ intersecting transversally $\R^n\times\{0\}$ in such a way that the relative Morse function is a height function. The existence of a good orientation is used to define the elementary cobordisms as in \ref{en:a}\&\ref{en:b} and to avoid the appearing of additional critical points as the realization of each edge $e\in E$ is transverse to the fibers of $\pi$, and hence to the fibers of $\Pi$ when regarding $\R^n$ as $\R^n\times\{0\}$ in $\R^{n+1}$. Indeed, let $a,b\in\N$ such that $a,b\leq 2$ and $a+b\in\{1,3\}$. Let $S_a:=\bigsqcup_{i=1}^a \mathbb{S}^{n-1}$ and $S_b:=\bigsqcup_{i=1}^b\mathbb{S}^{n-1}$. Embed in the standard way $S_a$ and $S_b$ in $\{-1\}\times\R^{n}\subset\R^{n+1}$ and $\{1\}\times\R^{n}\subset\R^{n+1}$ as spheres of radius $1$, respectively. Observe that both $S_a$ and $S_b$ are transverse to $\R^n\times\{0\}\subset\R^{n+1}$. Then, the elementary cobordism $M_{a,b}$ in \ref{en:b} is realized as half of a unit sphere $\mathbb{S}^{n}$ bounding either $S_a$ or $S_b$ if $a+b=1$. On the other hand, if $a+b=3$, then the elementary cobordism $M_{a,b}$ bounding $S_a\sqcup S_b$ is given by the intersection of the usual (algebraic) realization of a torus $\mathbb{S}^{n-1}\times\mathbb{S}^1$ as a hypersurface of $\R^{n+1}$ with $[-1,1]\times\R^n$. Therefore, in both cases, we may suppose the unique critical point of $\Pi|_{M_{a,b}}$ is contained in $\{0\}\times\R^{n-1}\times\{0\}$, $M_{a,b}$ is transverse to $\R^n\times\{0\}$ and $M_{a,b}$ intersects transversally $\{-1\}\times\R^n$ and $\{1\}\times\R^n$. Observe that if $a+b=3$, the critical locus $\Pi|_{M_{a,b}}^{-1}(0)$ of $\Pi|_{M_{a,b}}$ at $0$ is homeomorphic to $\mathbb{S}^{n-1}\vee\mathbb{S}^{n-1}$. 
	
	Let us construct then a $\cinfty$ hypersurface $M\subset\R^{n+1}$ satisfying (\ref{en:1})-(\ref{en:3}). Let $v\in V$ and define $\Gamma_v:=(V',E')$ as the graph such that $V':=V\setminus\{v\}$ and $E':=\{e\in E\,|\,v\notin e\}$. As the original graph $\Gamma$ is finite and without unbounded edges, then $\Gamma'$ is so and therefore $|\Gamma'|\subset|\Gamma|$ is a compact subset of $\R^n$ such that $v\notin |\Gamma'|$. Let $\delta_v>0$ such that $B^n(v,\delta_v)\cap|\Gamma'|=\varnothing$ and define $\delta:=\min_{v\in V}\delta_v>0$. Let $e\in E$ be such that $e=[v,w]$ for $v,w\in V$. As the realization $|e|$ of $e$ in $\R^n$ is a $\cinfty$ curve transverse to the fibers of $\pi$, let $\epsilon_e>0$ be sufficiently small such that $B^{n}(x,\epsilon_e)\cap e'=\varnothing$ for every $e'\in E$ with $e'\neq e$ and $x\in\pi^{-1}([\pi(v)+\frac{\delta}{2},\pi(w)-\frac{\delta}{2}])\cap |e|$. Observe that such a $\epsilon_e>0$ as above exists by compactness of $|\Gamma|$ and $|e|$. Then, define $\epsilon:=\frac{1}{3}\min_{e\in E}\epsilon _e$. Let $x,y\in |\Gamma|\setminus(\bigcup_{v\in V} B^n(v,\frac{\delta}{2}))$. Therefore, by construction, $B^n(x,\epsilon)$ intersects the realization $|e|$ of an edge $e\in E$ if and only if $x\in |e|$, and $B^n(x,\epsilon)\cap B^n(y,\epsilon)=\varnothing$ if $x\in |e_x|$ and $y\in|e_y|$ for $e_x,e_y\in E$ such that $e_x\neq e_y$. Hence, up to shrinking again $\epsilon>0$ if necessary, we may realize the above disjoint unions of spheres $S_a$ and $S_b$ for every vertex $v\in V$ in such a way that $S_a\cup S_b\subset B^{n+1}(v,\delta)$ is a disjoint union of spheres of radius $\epsilon>0$, $S_a\subset  B^{n+1}(v,\delta)\cap\Pi^{-1}(\pi(v-\frac{\delta}{2}))$ while $S_b\subset  B^{n+1}(v,\delta)\cap\Pi^{-1}(\pi(v+\frac{\delta}{2}))$ and whose centers are $|e|\cap \pi^{-1}(\pi(v-\frac{\delta}{2}))$ or $|e|\cap\pi^{-1}(\pi(v+\frac{\delta}{2}))$, respectively, for every $e\in E$ having $v$ as a vertex. Observe that, as above $S_a$ and $S_b$ intersect transversally $\R^n\times\{0\}$. Denote by $N_e$ the boundary of an $\epsilon$-tubular neighborhood of $|e|$ in $\R^{n+1}$, for every $e\in E$. Then, realize the elementary cobordism $M_v:=M_{a,b}$ as above bounding $S_a\sqcup S_b$ in $B^{n+1}(v,\delta)\cap\Pi^{-1}(\pi([v-\frac{\delta}{2},v+\frac{\delta}{2}]))$ in such a way that the unique critical point of $\Pi|_{M_{a,b}}$ is $v$, $M_{a,b}$ intersects transversally $\R^n\times\{0\}$ and the germ of $M_{a,b}$ at its boundary equals the germ of $N_e$ at a connected component of $S_a\cup S_b$ for every $e\in E$ such that $v$ is a vertex of $e$. Observe that, as each edge is realized by a $\cinfty$ curve transverse to the fibers of $\pi$, then $\Pi|_{N_e}$ does not have any critical point. Therefore, define $M_e:=N_e\cap\Pi^{-1}([\pi(v)+\frac{\delta}{2},\pi(w)-\frac{\delta}{2}])$ for $e=[v,w]\in E$. Finally, it suffices define $M:=\bigcup_{v\in V} M_v\cup\bigcup_{e\in E}M_e$, which is a $\cinfty$ manifold satisfying (\ref{en:1})-(\ref{en:3}) by construction.
	
\begin{figure}[h!]\label{fig:michalak}
	
\tikzset{every picture/.style={line width=0.75pt}} 

\begin{tikzpicture}[x=0.75pt,y=0.75pt,yscale=-1,xscale=1]
	
	\draw    (73,219) -- (499,220) ;
	\draw    (73,219) -- (172,81) ;
	\draw    (598,82) -- (499,220) ;
	\draw [color={rgb, 255:red, 0; green, 0; blue, 0 }  ,draw opacity=1 ][line width=1.5]    (222,136) -- (280,136) ;
	\draw [shift={(280,136)}, rotate = 0] [color={rgb, 255:red, 0; green, 0; blue, 0 }  ,draw opacity=1 ][fill={rgb, 255:red, 0; green, 0; blue, 0 }  ,fill opacity=1 ][line width=1.5]      (0, 0) circle [x radius= 4.36, y radius= 4.36]   ;
	\draw [shift={(222,136)}, rotate = 0] [color={rgb, 255:red, 0; green, 0; blue, 0 }  ,draw opacity=1 ][fill={rgb, 255:red, 0; green, 0; blue, 0 }  ,fill opacity=1 ][line width=1.5]      (0, 0) circle [x radius= 4.36, y radius= 4.36]   ;
	\draw [line width=1.5]    (415,135) .. controls (427,118) and (484,117) .. (508,118) ;
	\draw [shift={(508,118)}, rotate = 2.39] [color={rgb, 255:red, 0; green, 0; blue, 0 }  ][fill={rgb, 255:red, 0; green, 0; blue, 0 }  ][line width=1.5]      (0, 0) circle [x radius= 4.36, y radius= 4.36]   ;
	\draw [line width=1.5]    (280,136) .. controls (300,103) and (380,102) .. (357,135) ;
	\draw [line width=1.5]    (280,136) .. controls (255,167) and (332,168) .. (357,135) ;
	\draw [line width=1.5]    (357,135) -- (415,135) ;
	\draw [shift={(415,135)}, rotate = 0] [color={rgb, 255:red, 0; green, 0; blue, 0 }  ][fill={rgb, 255:red, 0; green, 0; blue, 0 }  ][line width=1.5]      (0, 0) circle [x radius= 4.36, y radius= 4.36]   ;
	\draw [shift={(357,135)}, rotate = 0] [color={rgb, 255:red, 0; green, 0; blue, 0 }  ][fill={rgb, 255:red, 0; green, 0; blue, 0 }  ][line width=1.5]      (0, 0) circle [x radius= 4.36, y radius= 4.36]   ;
	\draw [line width=1.5]    (415,135) .. controls (400,154) and (429,168) .. (450,167) ;
	\draw [shift={(450,167)}, rotate = 357.27] [color={rgb, 255:red, 0; green, 0; blue, 0 }  ][fill={rgb, 255:red, 0; green, 0; blue, 0 }  ][line width=1.5]      (0, 0) circle [x radius= 4.36, y radius= 4.36]   ;
	\draw  [dash pattern={on 4.5pt off 4.5pt}]  (245,147) .. controls (201,144) and (221,126) .. (254,127) ;
	\draw [line width=1.5]    (222,136) .. controls (220,121) and (234,109) .. (251,109) ;
	\draw    (245,147) .. controls (244,103) and (254,98) .. (254,127) ;
	\draw  [dash pattern={on 4.5pt off 4.5pt}]  (245,147) .. controls (274,144) and (242,171) .. (299,171) ;
	\draw [line width=1.5]    (297,126) .. controls (296,134) and (317,140) .. (332,133) ;
	\draw    (299,171) .. controls (300,119) and (308,130) .. (308,151) ;
	\draw    (321,121) .. controls (320,72) and (331,75) .. (330,101) ;
	\draw [line width=1.5]    (251,109) .. controls (270,107) and (287,86) .. (325,83) ;
	\draw    (381,145) .. controls (381,93) and (391,101) .. (390,125) ;
	\draw  [dash pattern={on 4.5pt off 4.5pt}]  (299,171) .. controls (353,170) and (354,147) .. (381,145) ;
	\draw [line width=1.5]    (325,83) .. controls (371,84) and (361,110) .. (387,107) ;
	\draw [line width=1.5]    (303,133) .. controls (310,126) and (327,132) .. (319,136) ;
	\draw  [dash pattern={on 4.5pt off 4.5pt}]  (254,127) .. controls (270,129) and (275,106) .. (330,101) ;
	\draw  [dash pattern={on 4.5pt off 4.5pt}]  (330,101) .. controls (373,97) and (374,124) .. (390,125) ;
	\draw  [dash pattern={on 4.5pt off 4.5pt}]  (381,145) .. controls (403,148) and (392,162) .. (422,172) ;
	\draw  [dash pattern={on 4.5pt off 4.5pt}]  (471,106) .. controls (488,106) and (510,105) .. (508,118) .. controls (506,131) and (483,125) .. (462,126) ;
	\draw  [dash pattern={on 4.5pt off 4.5pt}]  (390,125) .. controls (401,125) and (441,104) .. (471,106) ;
	\draw  [dash pattern={on 4.5pt off 4.5pt}]  (357,135) .. controls (336,155) and (257,158) .. (280,136) ;
	\draw  [dash pattern={on 4.5pt off 4.5pt}]  (442,157) .. controls (393,137) and (427,131) .. (462,126) ;
	\draw  [dash pattern={on 4.5pt off 4.5pt}]  (442,157) .. controls (462,161) and (453,184) .. (422,172) ;
	\draw    (462,126) .. controls (463,72) and (471,82) .. (471,106) ;
	\draw  [dash pattern={on 4.5pt off 4.5pt}]  (280,136) .. controls (308,114) and (373,120) .. (357,135) ;
	\draw    (422,172) .. controls (423,113) and (432,130) .. (431,152) ;
	\draw [line width=1.5]    (387,107) .. controls (409,103) and (437,88) .. (467,87) ;
	\draw [line width=1.5]    (427,132) .. controls (414,127) and (418,122) .. (424,122) ;
	\draw [line width=1.5]    (427,132) .. controls (445,134) and (451,151) .. (450,167) ;
	\draw    (245,147) .. controls (245,173) and (254,167) .. (254,127) ;
	\draw    (299,171) .. controls (299,194) and (307,194) .. (308,151) ;
	\draw [line width=1.5]    (222,136) .. controls (218,152) and (238,162) .. (247,162) ;
	\draw    (321,121) .. controls (322,151) and (331,131) .. (330,101) ;
	\draw [line width=1.5]    (247,162) .. controls (265,163) and (244,186) .. (301,186) ;
	\draw [line width=1.5]    (301,186) .. controls (361,184) and (360,160) .. (383,160) ;
	\draw    (462,126) .. controls (462,153) and (471,143) .. (471,106) ;
	\draw    (381,145) .. controls (381,174) and (391,159) .. (390,125) ;
	\draw    (422,172) .. controls (422,201) and (431,187) .. (431,152) ;
	\draw [line width=1.5]    (383,160) .. controls (401,164) and (392,179) .. (424,187) ;
	\draw [line width=1.5]    (424,187) .. controls (439,190) and (448,183) .. (450,167) ;
	\draw [line width=1.5]    (508,118) .. controls (507,146) and (459,143) .. (445,142) ;
	\draw [line width=1.5]    (467,87) .. controls (483,88) and (508,93) .. (508,118) ;
	
	\draw (537,174.4) node [anchor=north west][inner sep=0.75pt]    {$\mathbb{R}^{n} \times \{0\}$};
	\draw (523,53.4) node [anchor=north west][inner sep=0.75pt]    {$\mathbb{R}^{n+1}$};
	\draw (110,170) node [anchor=north west][inner sep=0.75pt]    {$\partial\mathscr{D}=M\cap\left(\mathbb{R}^{n}\times\{0\}\right)$};
	\draw (465,150) node [anchor=north west][inner sep=0.75pt]    {$|\Gamma|$};
	\draw (349,63.4) node [anchor=north west][inner sep=0.75pt]    {$M$};
	\draw (295,192) node [anchor=north west][inner sep=0.75pt]    {$\mathbb{S}^{n-1}$};

\end{tikzpicture}
	\caption{Construction of $M$ as a hypersurface of $\R^{n+1}$ satisfying (\ref{en:1})-(\ref{en:3}) above.}
\end{figure}
	
	
	Let $h\in\cinfty(\R^{n+1})$ such that $M=\mathcal{Z}_{\R^{n+1}}(h)$ and $\nabla h(x,y)\neq (\overline{0},0)$ for every $(x,y)\in M$, where $x\in\R^n$ and $y\in\R$. Up to changing the sign of $h$ we may assume that $\{(x,y)\in\R^n\times\R\,|\,h(x,y)\geq 0\}$ is compact. Consider the set $\mscr{D}:=\{x\in\R^n\,|\,h(x,0)\geq 0\}$ and $\partial\mscr{D}:=\{x\in\R^n\,|\,h(x,0)=0\}$. Observe that $\partial\mscr{D}\subset\R^n$ is a compact $\cinfty$ hypersurface by (\ref{en:1}) and $\pi|_{\partial\mscr{D}}:\partial \mscr{D}\to \R$ satisfies $\Pi|_{\partial\mscr{D}\times\{0\}}=\pi|_{\partial\mscr{D}}\times 0$. Hence, $\pi|_{\partial\mscr{D}}$ is a Morse function by (\ref{en:2}) and its critical points do coincide with the critical points of $\Pi|_M$ by Remark \ref{rem:crit}. Observe that $\pi|_{\mscr{D}}^{-1}(z)$ has the same number of connected components of $\Pi|_{M}^{-1}(z)$ by \ref{en:a}\&(\ref{en:1}) for regular values $z\in \R$ and each of them is diffeomorphic to a disk $\mathbb{D}^n$. Let $z\in\R$ be a critical value of $\Pi|_{M}$. Again, $\pi|_{\mscr{D}}^{-1}(z)$ has the same number of connected components of $\Pi|_{M}^{-1}(z)$ which are diffeomorphic to $\mathbb{D}^n$ by \ref{en:a}\&(\ref{en:1}), thus it suffices to show that the intersection $D$ of the connected component $C$ of $\Pi|_{M}^{-1}(z)$ containing the (unique) critical point $v\in V$ with $\R^{n}\times\{0\}$ is connected. If the index of $v$ is 1, then $C=\{v\}=D$. If the index of $v$ is $3$, then $C$ is homeomorphic to $\mathbb{S}^{n-1}\vee\mathbb{S}^{n-1}$ as it corresponds to the critical locus of a saddle point and hence $D$ is homeomorphic to the union of two disks $\mathbb{D}^{n-1}$ intersecting in a point of their boundaries, therefore it is connected and also contractible. This shows that $\mathcal{R}(\mscr{D})\cong\mathcal{R}(\Pi|_{M})$ and hence $\Gamma\cong\mathcal{R}(\mscr{D})$ as oriented graphs. Finally, an application of Theorem \ref{thm:approx} and Proposition \ref{prop:equiv} gives the result. Moreover, to get \ref{PR1-2} it suffices to apply Proposition \ref{prop:hom-equiv} as each connected component of $\pi|_{\mscr{D}}^{-1}(z)$ is contractible for every $z\in\R$.
\end{proof}

\begin{remark}
	Observe that, as a consequence of Theorem \ref{thm:PR-1}, the condition stated in \cite[Thm.4.6]{BPS24} that the graph $G'$ has to be realizable by a connected globally algebraic domain of finite type $\mscr{D}'\subset\R^2$ is always satisfied; therefore, the conclusion of their result holds even without this condition.
\end{remark}

	Observe that in the case $n=2$ there are no domains of finite type whose Poincaré-Reeb graph has vertices of degree $2$. Hence, for $n=2$, a compact domain of finite type $\mscr{D}\subset\R^n$ is stable if and only if its Poincaré-Reeb graph $\mathcal{R}(\mscr{D})$ only has vertices of degree $1$ or $3$ with distinct images under $\pi$. In higher dimension the situation is different. The next result describes a class of examples of stable domains of finite type $\mscr{D}\subset\R^n$ with $n\geq 3$ whose Poincaré-Reeb graph has vertices of degree $2$. First of all let us introduce a combinatorial condition on the set of vertices of degree $2$ of a graph.

\begin{definition}\label{def:pairing}
	Let $\Gamma:=(V,E)$ be a graph (without unbounded edges) and denote by $V_2:=\{v\in V\,|\,\deg(v)=2\}$ the set of vertices of degree $2$. We say that $\Gamma$ satisfies the \emph{pairing and monotonicity condition} if:
	\begin{enumerate}[label=\emph{(\roman*)}, ref=(\roman*)]
		\item\label{PR2-1} the cardinality $\#V_2$ of $V_2$ is even.
		\item\label{PR2-2} Let $k\in\N$ such that $\#V_2=2k$. Then, there is a choice of elements $v_1,\dots,v_k\in V_2$ and a function $\phi:\{v_1,\dots,v_k\}\to\{v_{k+1},\dots,v_{2k}\}$ such that for every $i=1,\dots,k$ there is a continuous path $\gamma_i:[0,1]\to|\Gamma|$ such that $\gamma_i(0)=v_i$, $\gamma_i(1)=\phi(v_i)$ and $\pi\circ\gamma:[0,1]\to\R$ is monotone.
	\end{enumerate}
\end{definition}

See Figure \ref{fig:V2} below for examples of graphs that do and do not satisfy the parity and monotonicity condition. 

\begin{figure}[h!]\label{fig:V2}
	\tikzset{every picture/.style={line width=0.75pt}} 
	
	\begin{tikzpicture}[x=0.7pt,y=0.7pt,yscale=-1,xscale=1]
		
		\draw [line width=1.5]    (89,131) -- (140,131) ;
		\draw [shift={(140,131)}, rotate = 0] [color={rgb, 255:red, 0; green, 0; blue, 0 }  ][fill={rgb, 255:red, 0; green, 0; blue, 0 }  ][line width=1.5]      (0, 0) circle [x radius= 4.36, y radius= 4.36]   ;
		\draw [shift={(114.5,131)}, rotate = 0] [color={rgb, 255:red, 0; green, 0; blue, 0 }  ][fill={rgb, 255:red, 0; green, 0; blue, 0 }  ][line width=1.5]      (0, 0) circle [x radius= 4.36, y radius= 4.36]   ;
		\draw [shift={(89,131)}, rotate = 0] [color={rgb, 255:red, 0; green, 0; blue, 0 }  ][fill={rgb, 255:red, 0; green, 0; blue, 0 }  ][line width=1.5]      (0, 0) circle [x radius= 4.36, y radius= 4.36]   ;
		\draw [line width=1.5]    (228,131) -- (279,131) ;
		\draw [shift={(279,131)}, rotate = 0] [color={rgb, 255:red, 0; green, 0; blue, 0 }  ][fill={rgb, 255:red, 0; green, 0; blue, 0 }  ][line width=1.5]      (0, 0) circle [x radius= 4.36, y radius= 4.36]   ;
		\draw [shift={(253.5,131)}, rotate = 0] [color={rgb, 255:red, 0; green, 0; blue, 0 }  ][fill={rgb, 255:red, 0; green, 0; blue, 0 }  ][line width=1.5]      (0, 0) circle [x radius= 4.36, y radius= 4.36]   ;
		\draw [shift={(228,131)}, rotate = 0] [color={rgb, 255:red, 0; green, 0; blue, 0 }  ][fill={rgb, 255:red, 0; green, 0; blue, 0 }  ][line width=1.5]      (0, 0) circle [x radius= 4.36, y radius= 4.36]   ;
		\draw  [line width=1.5]  (140,131) .. controls (140,106.7) and (159.7,87) .. (184,87) .. controls (208.3,87) and (228,106.7) .. (228,131) .. controls (228,155.3) and (208.3,175) .. (184,175) .. controls (159.7,175) and (140,155.3) .. (140,131) -- cycle ;
		\draw [line width=1.5]    (158,168) ;
		\draw [shift={(158,168)}, rotate = 0] [color={rgb, 255:red, 0; green, 0; blue, 0 }  ][fill={rgb, 255:red, 0; green, 0; blue, 0 }  ][line width=1.5]      (0, 0) circle [x radius= 4.36, y radius= 4.36]   ;
		\draw [line width=1.5]    (511,130) -- (562,130) ;
		\draw [shift={(562,130)}, rotate = 0] [color={rgb, 255:red, 0; green, 0; blue, 0 }  ][fill={rgb, 255:red, 0; green, 0; blue, 0 }  ][line width=1.5]      (0, 0) circle [x radius= 4.36, y radius= 4.36]   ;
		\draw [shift={(511,130)}, rotate = 0] [color={rgb, 255:red, 0; green, 0; blue, 0 }  ][fill={rgb, 255:red, 0; green, 0; blue, 0 }  ][line width=1.5]      (0, 0) circle [x radius= 4.36, y radius= 4.36]   ;
		\draw [line width=1.5]    (372,130) -- (423,130) ;
		\draw [shift={(423,130)}, rotate = 0] [color={rgb, 255:red, 0; green, 0; blue, 0 }  ][fill={rgb, 255:red, 0; green, 0; blue, 0 }  ][line width=1.5]      (0, 0) circle [x radius= 4.36, y radius= 4.36]   ;
		\draw [shift={(372,130)}, rotate = 0] [color={rgb, 255:red, 0; green, 0; blue, 0 }  ][fill={rgb, 255:red, 0; green, 0; blue, 0 }  ][line width=1.5]      (0, 0) circle [x radius= 4.36, y radius= 4.36]   ;
		\draw  [line width=1.5]  (423,130) .. controls (423,105.7) and (442.7,86) .. (467,86) .. controls (491.3,86) and (511,105.7) .. (511,130) .. controls (511,154.3) and (491.3,174) .. (467,174) .. controls (442.7,174) and (423,154.3) .. (423,130) -- cycle ;
		\draw [line width=1.5]    (442,166) ;
		\draw [shift={(442,166)}, rotate = 0] [color={rgb, 255:red, 0; green, 0; blue, 0 }  ][fill={rgb, 255:red, 0; green, 0; blue, 0 }  ][line width=1.5]      (0, 0) circle [x radius= 4.36, y radius= 4.36]   ;
		\draw [line width=1.5]    (491,93) ;
		\draw [shift={(491,93)}, rotate = 0] [color={rgb, 255:red, 0; green, 0; blue, 0 }  ][fill={rgb, 255:red, 0; green, 0; blue, 0 }  ][line width=1.5]      (0, 0) circle [x radius= 4.36, y radius= 4.36]   ;
		
		\draw [line width=1.5]    (208,94) ;
		\draw [shift={(208,94)}, rotate = 0] [color={rgb, 255:red, 0; green, 0; blue, 0 }  ][fill={rgb, 255:red, 0; green, 0; blue, 0 }  ][line width=1.5]      (0, 0) circle [x radius= 4.36, y radius= 4.36]   ;
		
		\draw (377,51.4) node [anchor=north west][inner sep=0.75pt]    {$| \Gamma _{2} |$};
		\draw (92,52.4) node [anchor=north west][inner sep=0.75pt]    {$| \Gamma _{1} |$};
		\draw (108,140) node [anchor=north west][inner sep=0.75pt]    {$v_{1}$};
		\draw (145,178) node [anchor=north west][inner sep=0.75pt]    {$v_{2}$};
		\draw (210,73) node [anchor=north west][inner sep=0.75pt]    {$v_{3}$};
		\draw (247,140) node [anchor=north west][inner sep=0.75pt]    {$v_{4}$};
		\draw (430,178) node [anchor=north west][inner sep=0.75pt]    {$w_{1}$};
		\draw (492,73) node [anchor=north west][inner sep=0.75pt]    {$w_{2}$};

	\end{tikzpicture}
	\caption{$\Gamma_{1}$ satisfies the pairing and the monotonicity condition, while $\Gamma_{2}$ does not.}
\end{figure}

\begin{remark}\label{rem:pairing}
	Observe that if a graph $\Gamma:=(V,E)$ with $V_2\neq\varnothing$ satisfies the pairing and monotonicity condition, then a function $\phi$ as in Definition \ref{def:pairing} is, in general, not unique. Consider the graph $\Gamma_1$ whose realization is $|\Gamma_1|$ of Figure \ref{fig:V2}. Then, both $\phi_1:\{v_1,v_3\}\to\{v_2,v_4\}$ defined as $\phi_1(v_i)=v_{i+1}$ for $i\in\{1,3\}$ and $\phi_2:\{v_1,v_2\}\to\{v_3,v_4\}$ defined as $\phi_1(v_i)=v_{i+2}$ for $i\in\{1,2\}$ satisfy Definition \ref{def:pairing}\ref{PR2-2}.
\end{remark}

Then our result reads as follows:

\begin{theorem}\label{thm:Poincaré-Reeb2}
	Let $\Gamma:=(V,E)$ be a connected finite graph (without unbounded edges) whose vertices have degree $\leq 3$ such that $|\Gamma|\subset\R^n$ and $\pi|_{|\Gamma|}:|\Gamma|\to \R$ is a good orientation with $\pi(v)\neq\pi(w)$ for every $v,w\in V$. Assume in addition that $\Gamma$ satisfies the pairing and monotonicity condition.
	
	Then, for every $n\geq 3$ there exists a globally algebraic domain of finite type $\mscr{D}\subset\R^n$ whose Poincaré-Reeb graph $\mathcal{R}(\mscr{D})$ is isomorphic to $\Gamma$. In addition, $\partial\mscr{D}$ can be chosen to be a projectively $\Q$-closed $\Q$-nonsingular $\Q$-algebraic hypersurface of $\R^n$.
\end{theorem}

\begin{proof}
	Let us prove by induction on the cardinality of the set $V_2$ satisfying \ref{PR2-1}\&\ref{PR2-2}. If $\{v\in V\,|\,\deg(v)=2\}=\varnothing$ the statement reduces to Theorem \ref{thm:PR-1}. Fix $k\in\N$. Let $\Gamma:=(V,E)$ be a finite graph (without unbounded edges) such that $|\{v\in V\,|\,\deg(v)=2\}|=2(k+1)$. Observe that, since $\Gamma$ admits a good orientation, then for each vertex of $v\in V_2$ there are $w_1,w_2\in V$ with $w_1\neq w_2$ such that $[v,w_1],[v,w_2]\in E$. Consider the graph $\Gamma':=(V',E')$ obtained by removing two vertices of degree $2$. More precisely, let $v_1\in V_2$, define $V'=V\setminus\{v_1,v_2:=\phi(v_1)\}$ and $E'$ defined as follows:
	\begin{enumerate}
		\item\label{en:PR2-1} If $e:=[v_1,\phi(v_1)]\in E$, then fix $E':=(E\setminus \{e, e_1,e_2\})\cup\{e'\}$ where $e_i=[v_i,w_i]$ for some $w_1,w_2\in V$ with $w_1\neq w_2$ and $e'$ to be an additional edge connecting $w_1$ and $w_2$. 
		\item\label{en:PR2-2} If $e:=[v_1,\phi(v_1)]\notin E$, then let $w_1,w_2\in V$ and $w_3,w_4\in V$, with $w_1\neq w_2$ and $w_3\neq w_4$, be the unique elements such that $e_1:=[v_1,w_1],e_2:=[v_1,w_2],e_3:=[\phi(v_1),w_3],e_4:=[\phi(v_1),w_4]\in E$ and fix $E':=(E\setminus \{e_1,e_2,e_3,e_4\})\cup\{e'_1,e'_2\}$, where $e_1'$ and $e_2'$ are two edges connecting $w_1$ to $w_2$ and $w_3$ to $w_4$, respectively.
	\end{enumerate}
	
	Consider a realization $|\Gamma|\subset\R^n$ for $n\geq 3$ such that $\pi|_{|\Gamma|}$ induces a good orientation on $\Gamma$, each edge $e\in E$ is transverse to the fibers of $\pi$ and $\pi(v)=\pi(w)$ if and only if $v=w$ for every $v,w\in V$. Observe that, since $\Gamma$ admits a good orientation induced by $\pi$ on $|\Gamma|$ as in Lemma \ref{cor:orientation}, then $\Gamma'$ as above admits a good orientation induced by $\pi$ on $|\Gamma'|$ as well. Up to rename $v_1$ with $\phi(v_1)$, we may assume $\pi(v_1)<\pi(\phi(v_1))$. In particular, this ensures that $\pi(w_1)<\pi(v_1)<\pi(\phi(v_1))<\pi(w_2)$ in (\ref{en:PR2-1}) and $\pi(w_1)<\pi(v_1)<\pi(w_2)\leq\pi(w_3)<\pi(\phi(v_1))<\pi(w_4)$ with $\pi(w_2)=\pi(w_3)$ if and only if $w_2=w_3$ in (\ref{en:PR2-2}). 
	
	\begin{figure}\label{fig:V3}

		\tikzset{every picture/.style={line width=0.75pt}} 
		
		\begin{tikzpicture}[x=0.75pt,y=0.75pt,yscale=-1,xscale=1]
			
			\draw [line width=1.5]    (93,151) -- (144,151) ;
			\draw [shift={(144,151)}, rotate = 0] [color={rgb, 255:red, 0; green, 0; blue, 0 }  ][fill={rgb, 255:red, 0; green, 0; blue, 0 }  ][line width=1.5]      (0, 0) circle [x radius= 4.36, y radius= 4.36]   ;
			\draw [shift={(118.5,151)}, rotate = 0] [color={rgb, 255:red, 0; green, 0; blue, 0 }  ][fill={rgb, 255:red, 0; green, 0; blue, 0 }  ][line width=1.5]      (0, 0) circle [x radius= 4.36, y radius= 4.36]   ;
			\draw [shift={(93,151)}, rotate = 0] [color={rgb, 255:red, 0; green, 0; blue, 0 }  ][fill={rgb, 255:red, 0; green, 0; blue, 0 }  ][line width=1.5]      (0, 0) circle [x radius= 4.36, y radius= 4.36]   ;
			\draw [line width=1.5]    (232,151) -- (283,151) ;
			\draw [shift={(283,151)}, rotate = 0] [color={rgb, 255:red, 0; green, 0; blue, 0 }  ][fill={rgb, 255:red, 0; green, 0; blue, 0 }  ][line width=1.5]      (0, 0) circle [x radius= 4.36, y radius= 4.36]   ;
			\draw [shift={(257.5,151)}, rotate = 0] [color={rgb, 255:red, 0; green, 0; blue, 0 }  ][fill={rgb, 255:red, 0; green, 0; blue, 0 }  ][line width=1.5]      (0, 0) circle [x radius= 4.36, y radius= 4.36]   ;
			\draw [shift={(232,151)}, rotate = 0] [color={rgb, 255:red, 0; green, 0; blue, 0 }  ][fill={rgb, 255:red, 0; green, 0; blue, 0 }  ][line width=1.5]      (0, 0) circle [x radius= 4.36, y radius= 4.36]   ;
			\draw  [line width=1.5]  (144,151) .. controls (144,126.7) and (163.7,107) .. (188,107) .. controls (212.3,107) and (232,126.7) .. (232,151) .. controls (232,175.3) and (212.3,195) .. (188,195) .. controls (163.7,195) and (144,175.3) .. (144,151) -- cycle ;
			\draw [line width=1.5]    (162,188) ;
			\draw [shift={(162,188)}, rotate = 0] [color={rgb, 255:red, 0; green, 0; blue, 0 }  ][fill={rgb, 255:red, 0; green, 0; blue, 0 }  ][line width=1.5]      (0, 0) circle [x radius= 4.36, y radius= 4.36]   ;
			\draw [line width=1.5]    (515,150) -- (566,150) ;
			\draw [shift={(566,150)}, rotate = 0] [color={rgb, 255:red, 0; green, 0; blue, 0 }  ][fill={rgb, 255:red, 0; green, 0; blue, 0 }  ][line width=1.5]      (0, 0) circle [x radius= 4.36, y radius= 4.36]   ;
			\draw [shift={(540.5,150)}, rotate = 0] [color={rgb, 255:red, 0; green, 0; blue, 0 }  ][fill={rgb, 255:red, 0; green, 0; blue, 0 }  ][line width=1.5]      (0, 0) circle [x radius= 4.36, y radius= 4.36]   ;
			\draw [shift={(515,150)}, rotate = 0] [color={rgb, 255:red, 0; green, 0; blue, 0 }  ][fill={rgb, 255:red, 0; green, 0; blue, 0 }  ][line width=1.5]      (0, 0) circle [x radius= 4.36, y radius= 4.36]   ;
			\draw [line width=1.5]    (376,150) -- (427,150) ;
			\draw [shift={(427,150)}, rotate = 0] [color={rgb, 255:red, 0; green, 0; blue, 0 }  ][fill={rgb, 255:red, 0; green, 0; blue, 0 }  ][line width=1.5]      (0, 0) circle [x radius= 4.36, y radius= 4.36]   ;
			\draw [shift={(376,150)}, rotate = 0] [color={rgb, 255:red, 0; green, 0; blue, 0 }  ][fill={rgb, 255:red, 0; green, 0; blue, 0 }  ][line width=1.5]      (0, 0) circle [x radius= 4.36, y radius= 4.36]   ;
			\draw  [line width=1.5]  (427,150) .. controls (427,125.7) and (446.7,106) .. (471,106) .. controls (495.3,106) and (515,125.7) .. (515,150) .. controls (515,174.3) and (495.3,194) .. (471,194) .. controls (446.7,194) and (427,174.3) .. (427,150) -- cycle ;
			\draw [line width=1.5]    (495,113) ;
			\draw [shift={(495,113)}, rotate = 0] [color={rgb, 255:red, 0; green, 0; blue, 0 }  ][fill={rgb, 255:red, 0; green, 0; blue, 0 }  ][line width=1.5]      (0, 0) circle [x radius= 4.36, y radius= 4.36]   ;
			\draw [line width=1.5]    (212,114) ;
			\draw [shift={(212,114)}, rotate = 0] [color={rgb, 255:red, 0; green, 0; blue, 0 }  ][fill={rgb, 255:red, 0; green, 0; blue, 0 }  ][line width=1.5]      (0, 0) circle [x radius= 4.36, y radius= 4.36]   ;
			
			\draw (381,71.4) node [anchor=north west][inner sep=0.75pt]    {$| \Gamma '|$};
			\draw (96,72.4) node [anchor=north west][inner sep=0.75pt]    {$| \Gamma |$};
			\draw (109,157.4) node [anchor=north west][inner sep=0.75pt]    {$v_{1}$};
			\draw (75,157.4) node [anchor=north west][inner sep=0.75pt]    {$w_{1}$};
			\draw (150,140) node [anchor=north west][inner sep=0.75pt]    {$w_{2}=w_3$};
			\draw (208,157.4) node [anchor=north west][inner sep=0.75pt]    {$w_{4}$};
			\draw (153,196.4) node [anchor=north west][inner sep=0.75pt]    {$v_{2}$};
			\draw (210,95) node [anchor=north west][inner sep=0.75pt]    {$v_{3}$};
			\draw (249,158.4) node [anchor=north west][inner sep=0.75pt]    {$v_{4}$};
			\draw (531,159.4) node [anchor=north west][inner sep=0.75pt]    {$v_{4}$};
			\draw (493,95) node [anchor=north west][inner sep=0.75pt]    {$v_{3}$};

		\end{tikzpicture}
		\caption{Construction of $\Gamma':=(V',E')$ from $\Gamma:=\Gamma_1$ with function $\phi:=\phi_1$ such that $\phi(v_1)=v_2$ as in Remark \ref{rem:pairing}.}
	\end{figure}
	
	By inductive assumption, since $|V'_2|=k$, there exists a stable domain of finite type $\mscr{D}'\subset\R^n$ whose Poincaré-Reeb graph $\mathcal{R}(\mscr{D'})$ is isomorphic to $\Gamma'$. Choose two points $x_1,x_2\in\mscr{D}'\setminus\partial\mscr{D}'$ contained in the same connected component of $\mscr{D}'\cap\pi^{-1}(\pi([v_1,\phi(v_1)]))$ as $v_1$ and $\phi(v_1)$, respectively, such that $\pi(x_1)=\pi(v_1)$ and $\pi(x_2)=\pi(\phi(v_1))$. Construct a $\cinfty$ path $\gamma':[0,1]\to\mscr{D}'\setminus\partial\mscr{D}'$ such that: $\gamma'(0)=x_1$, $\gamma'(1)=x_2$, $\pi\circ\gamma':[0,1]\to\R$ is monotone increasing and $\gamma'$ is transverse to the fibers of $\pi$. Observe that the existence of $\gamma'$ is ensured by \ref{PR2-2}. Then, consider the boundary $M\subset\mscr{D}'\setminus\partial\mscr{D}'$ of a sufficiently small $\cinfty$ tubular neighborhood $N$ of $\gamma'([0,1])$ in $\mscr{D}'\setminus\partial\mscr{D}'$ such that $\pi|_{M}$ is a Morse function with just two critical points exactly at $x_1$ end $x_2$. Observe that the latter property is ensured by transversality of $\gamma'$ with respect to the fibers of $\pi:\R^n\to \R$.
	
	Let $f,g\in\cinfty(\R^n)$ such that $\partial\mscr{D}'=\mathcal{Z}_{\R^n}(f)$, $M=\mathcal{Z}_{\R^n}(g)$, $\nabla f(x)\neq \overline{0}$ for every $x\in\partial\mscr{D}'$ and $\nabla g(x)\neq \overline{0}$ for every $x\in M$, such that $\mscr{D}'=\{x\in\R^n\,|\,f(x)\geq 0\}$ and $\{x\in\R^n\,|\,-g(x)\geq 0\}\subset\mscr{D}'$. Define $h:=fg\in\cinfty(\R^n)$. Then, $\mscr{D}:=\{x\in\R^n\,|\,h(x)\geq 0\}$ is a stable domain of finite type such that $\partial\mscr{D}=\partial\mscr{D}'\cup M$ and, by construction, the Poincaré-Reeb graph $\mathcal{R}(\mscr{D})$ of $\mscr{D}\subset\R^n$ is isomorphic to $\Gamma$ in the sense of Definition \ref{def:iso-graph}. We conclude the proof by applying Theorem \ref{thm:approx} to $\mscr{D}$.
\end{proof}

We add a couple of remarks concerning Theorem \ref{thm:Poincaré-Reeb2}.

\begin{remark}\label{rem:construction}
	Assume $\Gamma:=(V,E)$ satisfies the hypothesis of Theorem \ref{thm:Poincaré-Reeb2} and $V_2:=\{v\in V\,|\,\deg(v)=2\}\neq \varnothing$. Then:
	\begin{enumerate}[label={(\roman*)}, ref=(\roman*)]
		\item As observed in Remark \ref{rem:pairing}, the function $\phi$ in Definition \ref{def:pairing} is in general not unique; therefore, different choices of $\phi$ in the proof of Theorem \ref{thm:Poincaré-Reeb2} can give rise to compact domains of finite type $\mscr{D},\mscr{D'}\subset \R^n$ that are not vertically equivalent to each other. This shows that, the converse implication of Proposition \ref{prop:equiv}, namely, $\mathcal{R}(\mscr{D})\cong\mathcal{R}(\mscr{D}')$ implies $\mscr{D}\sim_v \mscr{D}'$ does not hold even for stable compact domains of finite type $\mscr{D},\mscr{D'}\subset\R^n$, with $n\geq 3$. Indeed, if two stable domains of finite type $\mscr{D}, \mscr{D}'\subset \R^n$ such that $\mathcal{R}(\mscr{D})=\mathcal{R}(\mscr{D}')$ are vertically equivalent, then $\pi^{-1}(c)\cap\mscr{D}\cong \pi^{-1}(c)\cap\mscr{D}'$ for every $c\in\R$. Consider the graph $\Gamma_{1}$ and its realization $|\Gamma_{1}|$ in Figure \ref{fig:V2}. Let $\mscr{D},\mscr{D}'\subset\R^n$ be the domains of finite type that we obtain as an application of the proof of Theorem \ref{thm:Poincaré-Reeb2} by choosing the functions $\phi_1$ and $\phi_2$ of Remark \ref{rem:pairing}, respectively. Denote by $\mathbb{D}^n$ the closed unit $n$-disk and by $\mathbb{D}_1,\mathbb{D}_2\subset \text{int} \mathbb{D}^n$ be two closed $n$-disks such that $\mathbb{D}_1\cap \mathbb{D}_2=\varnothing$. Then, by following the argument in the proof of Theorem \ref{thm:Poincaré-Reeb2}, we observe that
		\[
		\pi^{-1}(c)\cap\mscr{D}\cong \mathbb{D}^n\sqcup \mathbb{D}^n\not\cong \pi^{-1}(c)\cap\mscr{D}'
		\] 
		for every $c\in]\pi(v_2),\pi(v_3)[$, as $\pi^{-1}(c)\cap\mscr{D}'$ is either diffeomorphic to $(\mathbb{D}^n\setminus \text{int} \mathbb{D}_1)\sqcup ( \mathbb{D}^n\setminus \text{int}  \mathbb{D}_2)$ or to $( \mathbb{D}^n\setminus (\text{int}  \mathbb{D}_1\cup\text{int}  \mathbb{D}_2))\sqcup  \mathbb{D}^n$ depending on the choices of the smooth paths $\gamma_1$ and $\gamma_2$ (as in the proof of Theorem \ref{thm:Poincaré-Reeb2}) connecting $x_1$ and $x_3$, and $x_2$ and $x_4$, respectively.
		
		\item Observe that, in general, the domains of finite type $\mscr{D}\subset\R^{n}$ with $n\geq 3$ obtained as in the proof of Theorem \ref{thm:Poincaré-Reeb2} are not homotopically equivalent to their Poincaré-Reeb space $\widetilde{\mscr{D}}$ when $\{v\in V\,|\,\deg(v)=2\}\neq \varnothing$. As an evidence, observe that $\mscr{D}\subset\R^n$ in Example \ref{ex:non-retract} can be obtained by following the proof of Theorem \ref{thm:Poincaré-Reeb2} by starting with the standard unit disk $\mathbb{D}^n\subset\R^n$ and $\mscr{D}\cong \mathbb{S}^{n-1}\times[0,1]$, which is not contractible, while its Poincaré-Reeb space $\widetilde{\mscr{D}}$ is, therefore $\mscr{D}$ and $\widetilde{\mscr{D}}$ are not homotopically equivalent.
	\end{enumerate}
\end{remark}

We conclude the article by extending Theorems \ref{thm:PR-1}\&\ref{thm:Poincaré-Reeb2} to any real closed field $R$.

\begin{remark}\label{rem:final}
	Observe that Theorems \ref{thm:PR-1}\&\ref{thm:Poincaré-Reeb2} hold over every real closed field $R$, even if the approximation techniques are not always available. In addition, the rational numbers usually are not (Euclidean) dense over a general real closed field $R$, so even when approximation techniques are available it is not guaranteed to get polynomial equations over $\Q$. Let us outline the way of proving it in the case of Theorem \ref{thm:PR-1}, the case of Theorem \ref{thm:Poincaré-Reeb2} is similar.
	\begin{enumerate}
		\item \emph{Extend the notion of domain of finite $\mscr{D}$ type over $R$.} A subset $\mscr{D}\subset R^n$ is a domain of finite type if $\mscr{D}$ is a Nash submanifold of dimension $n$ with Nash boundary $\partial\mscr{D}\subset R^n$ satisfying the properties Definition \ref{def:dom}\ref{en:dom1}\&\ref{en:dom2} with the following substitutions: $\pi|_{\partial\mscr{D}}$ is Nash Morse function, that is, a Nash function having finitely many critical points and each of them is nondegenerate (i.e. after a Nash change of coordinates the Hessian of $\pi|_{\partial\mscr{D}}$ at a critical point $p$ is a non-singular symmetric matrix) and we substitute the property of properness of $\pi|_{\mscr{D}}$ with the following: the preimage under $\pi|_{\mscr{D}}$ of a closed and bounded semialgebraic set is bounded. Observe that, as in Remark \ref{rem:implicite}(i), the set of critical points of $\pi|_{\partial\mscr{D}}$ corresponds to the set defined by $\frac{\partial f}{\partial x_i}(x)=0$ for every $i=2,\dots,n$, where $f\in \mathcal{N}(R^n)$ is a Nash function whose gradient never vanish over $\partial \mscr{D}=\mathcal{Z}(h)$. In particular, the set of critical points of $\pi|_{\partial\mscr{D}}$ is semialgebraic. By Definition \ref{def:dom}\ref{en:dom1} the set of critical values of $\pi|_{\partial\mscr{D}}$ is finite and $\pi|_{\partial\mscr{D}}$ is Morse, thus the set of critical points of $\pi|_{\partial\mscr{D}}$ is discrete and semialgebraic, hence finite. For more details on Nash Morse functions over real closed fields, or more generally definable over o-minimal structures, we refer to \cite[\S 4.2]{PS07} and references therein.
		
		\item \emph{Extend the notion of Poincaré-Reeb graph $\mathcal{R}(\mscr{D})$ of a domain of finite type $\mscr{D}$ over $R$.} Consider the definition of the Poincaré-Reeb space $\widetilde{\mscr{D}}$ of $\mscr{D}$ in Definition \ref{def:P-R-space} and replace the equivalence relation $\sim$  on $\mscr{D}$ by the following: for every $x,y\in\mscr{D}$ we say that $x\sim y$ if and only if $x$ and $y$ lie in the same semialgebraically connected component of $\pi^{-1}(\pi(x))\cap\mscr{D}$. Then, the graph structure of $\widetilde{\mscr{D}}$ follows as in Definition \ref{def:PR-graph}. Observe that, when $R=\R$, the original definition we gave and this second one do coincide, see \cite[Thm.2.4.5]{BCR98}.
		
		\item Let $h\in\Q[x_1,\dots,x_n]$ satisfying Theorem \ref{thm:PR-1}. We show that the Poincaré-Reeb graph $\mathcal{R}(\mscr{D}_R)\subset R^n$ of $\mscr{D}_R:=\{x\in R^n\,|\, h(x)\geq 0\}$ is isomorphic to $\Gamma:=(V_R,E_R)$. Denote by $V_R\subset R^n$ the set of vertices and by $E_R\subset \mathcal{P}(R^n)$ the set of edges of $\mathcal{R}(\mscr{D}_R)$.
		\begin{enumerate}[label=(\alph*), ref=(\alph*)]
			\item \emph{The set of vertices $V_R\subset R^n$ of $\mathcal{R}(\mscr{D}_R)$ is contained in $(\overline{\Q}^r)^n:=(\R\cap\overline{\Q})^n$.} Observe that $V_R:=\big\{x\in R^n\,|\, h(x)=0,\frac{\partial h}{\partial x_i}(x)=0\quad\text{for $i=2,\dots,n$}\big\}$. Since $V_R$ is a $\Q$-algebraic subset of $R^n$ and $\overline{\Q}^r$ is a real closed field, by model completeness of the theory of real closed fields \cite{Tar51}, we have that $V_R=\textnormal{Zcl}_{R^n}(V_R\cap(\overline{\Q}^r)^n)$, where $\textnormal{Zcl}_{R^n}(V_R\cap(\overline{\Q}^r)^n)$ denotes the Zariski closure of $V_R\cap(\overline{\Q}^r)^n$ in $R^n$. Since $V_R$ is finite, then $V_R=V_R\cap(\overline{\Q}^r)^n$, that is, $V_R\subset(\overline{\Q}^r)^n$ for every real closed field $R$. This shows that the set of critical points $V_R$ of $\pi|_{\partial\mscr{D}_R}$ coincides with the set of critical points $V:=V_\R$ in Theorem \ref{thm:PR-1}.
			
			\item \emph{The set of edges $E_R\subset \mathcal{P}(R^n)$ of $\mathcal{R}(\mscr{D}_R)$ is a semialgebraic set defined over $\overline{\Q}^r:=\R\cap\overline{\Q}$.} Consider the set of critical values $C_R\in R$ of the projection $\pi|_{\partial\mscr{D}_R}$. Observe that, since the set of critical points $V_R\subset (\overline{\Q}^r)^n$ and $\pi:R^n\to R$ is the projection onto the first coordinate, then the set $C_R$ of critical values of $\pi|_{\partial\mscr{D}_R}$ satisfies $C_R=\pi|_{\partial\mscr{D}_R}(V_R)\subset \overline{\Q}^r$. Let $C_R=\{c_1,\dots,c_k\}\subset\overline{\Q}^r$ with $c_i<c_{i+1}$ for every $i\in\{1,\dots,k-1\}$. Let $d_i\in\overline{\Q}^r$ satisfying $c_i<d_i<c_{i+1}$ for every $i\in\{1,\dots,k-1\}$. Then, the set $\pi|_{\partial\mscr{D}_{\overline{\Q}^r}}^{-1} (]c_i,c_{i+1}[_{\overline{\Q}^r})=\pi|_{\partial\mscr{D}_R}^{-1}(]c_i,c_{i+1}[_R)\cap(\overline{\Q}^r)^n$ admits a semialgebraic trivialization $\varphi_{\overline{\Q}^r}:\pi|_{\partial\mscr{D}_{\overline{\Q}^r}}^{-1}(]c_i,c_{i+1}[_{\overline{\Q}^r}) \to \pi|_{\partial\mscr{D}_{\overline{\Q}^r}}^{-1}(d_i)\times ]c_i,c_{i+1}[_{\overline{\Q}^r}$ (\cite[Thm.A.5]{PS07}), that is, a semialgebraic homeomorphism $\varphi_{\overline{\Q}^r}$ as above satisfying $\pi\circ\varphi_{\overline{\Q}^r}=\pi$.  By model completeness of the theory of real closed fields \cite{Tar51}, the semialgebraic trivialization $\varphi_{\overline{\Q}^r}:\pi|_{\partial\mscr{D}_{\overline{\Q}^r}}^{-1}(]c_i,c_{i+1}[_{\overline{\Q}^r}) \to \pi|_{\partial\mscr{D}_{\overline{\Q}^r}}^{-1}(d_i)\times ]c_i,c_{i+1}[_{\overline{\Q}^r}$ extends to a semialgebraic trivialization $\varphi_R:\pi|_{\partial\mscr{D}_R}^{-1}(]c_i,c_{i+1}[_R) \to \pi|_{\partial\mscr{D}_R}^{-1}(d_i)\times ]c_i,c_{i+1}[_R$, showing that $E_R=E_{\overline{\Q}^r}$ for every real closed field $R$. In particular, $E_R=E_\R$ as in Theorem \ref{thm:PR-1}.
		\end{enumerate}
	\end{enumerate}
\end{remark}

\section*{Acknowledgments} The author would like to thank the anonymous referees for their careful reading of the manuscript and for their remarks, which have improved the presentation of this article.



	\printbibliography
	
\end{document}